\documentclass[a4paper,12pt,twoside,notitlepage,reqno]{amsart}

\usepackage{amsmath,amssymb,amsbsy,amsthm}
\usepackage[hmargin=1.4in,vmargin=1.4in]{geometry}
\usepackage{url}

\usepackage{tikz}
\usetikzlibrary{matrix}
\usepackage[linktocpage]{hyperref}
\hypersetup{
    colorlinks=true,        
    linkcolor=red,          
    citecolor=red,         
    filecolor=magenta,      
    urlcolor=cyan           
}

\addtocontents{toc}{\protect\setcounter{tocdepth}{1}}

\usepackage{eucal}

\newcommand{\Q}{ \mathbb{Q} }
\newcommand{\Z}{ \mathbb{Z} }
\newcommand{\N}{ \mathbb{N} }

\newcommand{\cl}[1]{ \overline{#1} }

\theoremstyle{plain}
	\newtheorem{thm}{Theorem}
	\newtheorem{notn}[thm]{Notation}
		\numberwithin{thm}{section}
	\newtheorem{lemma}[thm]{Lemma}
	\newtheorem{prop}[thm]{Proposition}
	\newtheorem{cor}[thm]{Corollary}

	\newtheorem*{thm*}{Theorem}
	\newtheorem*{lemma*}{Lemma}
	\newtheorem*{prop*}{Proposition}
	\newtheorem*{cor*}{Corollary}
	\newtheorem*{conj*}{Conjecture}

\theoremstyle{definition}
	\newtheorem{example}[thm]{Example}
	\newtheorem*{example*}{Example}
	\newtheorem{defn}[thm]{Definition}
	\newtheorem{remark}[thm]{Remark}

\begin{document}


\title{Rational dynamical systems, $S$-units, and $D$-finite power series}

\author{Jason P. Bell}
\address{Department of Pure Mathematics\\
University of Waterloo\\
Waterloo, ON N2L 3G1\\
Canada}
\email{jpbell@uwaterloo.ca}

\author{Shaoshi Chen}
\address{KLMM \\ Academy of Mathematics and Systems Science \\ Chinese Academy of Sciences \\
Beijing, 100190, China}
\email{schen@amss.ac.cn}

\author{Ehsaan Hossain}
\address{Department of Pure Mathematics\\
University of Waterloo\\
Waterloo, ON N2L 3G1\\
Canada}
\email{ehossain@uwaterloo.ca}

\begin{abstract}
Let $K$ be an algebraically closed field of characteristic zero and let $G$ be
a finitely generated subgroup of the multiplicative group of $K$.  We consider $K$-valued sequences of the form
$a_n:=f(\varphi^n(x_0))$, where $\varphi\colon X\dasharrow X$ and $f\colon X\dasharrow\mathbb{P}^1$ are rational maps
defined over $K$ and $x_0\in X$ is a point whose forward orbit avoids the indeterminacy loci
of $\varphi$ and $f$.  Many classical sequences from number theory and algebraic combinatorics fall under this dynamical
framework, and we show that the set of $n$ for which $a_n\in G$ is a finite union of arithmetic progressions along with
a set of Banach density zero.  In addition, we show that if $a_n\in G$ for every $n$ and $X$ is irreducible and the $\varphi$ orbit
of $x$ is Zariski dense in $X$ then there are a multiplicative torus $\mathbb{G}_m^d$ and maps $\Psi:\mathbb{G}_m^d \to \mathbb{G}_m^d$
and $g:\mathbb{G}_m^d \to \mathbb{G}_m$ such that $a_n = g\circ \Psi^n(y)$ for some $y\in \mathbb{G}_m^d$. We then obtain results about the coefficients of 
$D$-finite power series using these facts.
\end{abstract}

\maketitle

\tableofcontents

\section{Introduction}

A \textit{rational dynamical system} is a pair $(X,\varphi)$, where $X$ is a quasiprojective variety defined over a field $K$, and $\varphi:X\dashrightarrow X$ is a rational map. The forward $\varphi$-orbit of a point $x_0\in X$ is given by
\[ O_\varphi(x_0):=\{x_0,\varphi(x_0),\varphi^2(x_0),\ldots\} \]
 as long as this orbit is defined (\textit{i.e.}, $x_0$ is outside the indeterminacy locus of $\varphi^n$ for every $n\geq 0$).

In \cite{BGS} and \cite{BFS}, the authors develop a broad dynamical framework giving rise to many
classical sequences from number theory and algebraic combinatorics,
by considering \emph{dynamical sequences}, which are sequences of the form $ f\circ \varphi^n(x_0)$, where $(X,\varphi)$ is a rational
dynamical system, $f:X\dashrightarrow \mathbb{P}^1$ is a rational map, and $x_0\in X$.
In particular, the class of dynamical sequences includes all sequences whose generating functions are $D$-finite, i.e., those satisfying homogeneous linear differential
equations with rational function coefficients.  This is an important class of power series since it
appears ubiquitously in algebra, combinatorics, and number theory. In particular, this class contains:
\begin{itemize}
\item all hypergeometric series (see, for example, \cite{G09, WZ});
\item generating functions for many classes of lattice walks \cite{DHRS18};
\item diagonals of rational functions \cite{Lipshitz88};
\item power series expansions of algebraic functions \cite[Chapter 6]{Stan};
\item generating series for the cogrowth of many finitely presented groups \cite{GP};
\item many classical combinatorial sequences (see Stanley \cite[Chapter 6]{Stan} and Mishna \cite[Chapter 5]{Mishna20} for more examples).
\end{itemize}
  The $D$-finiteness of generating functions reflects the complexity of combinatorial classes~\cite{Pak18}.
Since this class is closed under addition, multiplication, and the process of taking diagonals, it has become a useful data structure
for the manipulation of special functions in symbolic computation~\cite{Salvy19}. The main goal of this paper is to
study $D$-finite power series in the framework of rational dynamical systems.

The study of power series with coefficients in a finitely generated subgroup $G$ of the multiplicative group of
a field enjoys a long history, going back at least to the early 1920s, with the pioneering work of P\'olya \cite{Pol},
who characterized rational functions whose Taylor expansions at the origin have coefficients lying in a finitely generated multiplicative subgroup of $\Z$.
P\'olya's results were later extended to $D$-finite power series by B\'ezivin \cite{Bez}.  From this point of view,
it is then natural to consider when the dynamical sequences we study take values in a finitely generated multiplicative group.
This question and related questions have already been considered in the case of self-maps of $\mathbb{P}^1$ in \cite{BOSS, BOSS2}.

We show that, with the above notation, the set of $n$ for which $f\circ \varphi^n(x_0)\in G$ is well-behaved.  To make this precise, we
recall that the \textit{Banach density} of a subset $S\subseteq \N_0$ is
\[ \delta(S):=\limsup_{|I|\rightarrow\infty}\frac{|S\cap I|}{|I|}, \]
where $I$ ranges over all intervals.  Then our main result is the following.

\begin{thm}
\label{main theorem}
Let $X$ be a quasiprojective variety over a field $K$ of characteristic zero, let $\varphi:X\dashrightarrow X$ be a rational map, let $f:X\dashrightarrow K$ be a rational function, and let $G\leq K^*$ be a finitely generated group. If $x_0\in X$ is a point with well-defined forward $\varphi$-orbit that also avoids the indeterminacy locus of $f$, then the set
\[ \left\{n\in \N_0 : f(\varphi^n(x_0)) \in G\right\} \]
is a finite union of arithmetic progressions along with a set of Banach density zero.
\end{thm}

In the case when $G$ is the trivial group, Theorem \ref{main theorem} follows independently from the work of several authors \cite{BGT, Gig, Petsche}, and was originally conjectured by 
Denis \cite{Denis}, and is now considered a relaxed version of the so-called Dynamical Mordell-Lang Conjecture (see \cite{DML-book} for further background). 
There is also some overlap between our result and that of Ghioca and Nguyen \cite{GN}: in particular, when $G$ is an infinite cyclic group and $f$ is a linear map, 
one gets a special case of their result.  
The zero density sets in Theorem \ref{main theorem} are not necessarily finite. For a simple example, 
consider the map $\varphi:\mathbb{A}^1\rightarrow\mathbb{A}^1$ defined by $\varphi(x)=x+1$ and the rational function $f(x) =x$. 
Let $G=\langle 2\rangle$. Then for the initial point $x_0=1$ the set $\{n\in \N_0 : f(\varphi^n(x_0))\in G\} = \{0,1,3,7,15,\ldots\}$ is 
an infinite set of Banach density zero.

The case $N=\N_0$ can be easily achieved: let $T:=\mathbb{G}_m^d$ be a $d$-dimensional multiplicative torus.  Then an endomorphism $\varphi$ of $T$ is a map of the form
$$(x_1,\ldots ,x_d)\mapsto \left(c_1\prod_{j} x_j^{a_{1,j}},\ldots ,c_d\prod_j x_j^{a_{d,j}}\right).$$
Then if we begin with a point  $x_0=(\beta_1,\ldots ,\beta_d)$, then every point in the orbit of $x_0$ has coordinates in the multiplicative group $G$ generated by $$c_1,\ldots ,c_d,\beta_1,\ldots, \beta_d.$$  In particular, if $f:T\to \mathbb{P}^1$ is a map of the form $(x_1,\ldots ,x_d)\mapsto \kappa x_1^{p_1}\cdots x_d^{p_d}$ with $\kappa\in G$ then $f\circ \varphi^n(x_0)\in G$ for every $n\ge 0$.

In fact, we show that in characteristic zero any dynamical system $(X,\varphi)$ with $N=\N_0$ is ``controlled'' by one of this form, in the following sense.

\begin{thm} \label{main theorem 2}
Let $K$ be an algebraically closed field of characteristic zero and let $X$ be an irreducible quasiprojective variety with a dominant self-map $\varphi:X\to X$ and
let $f:  X \to \mathbb{P}^1$ be a dominant rational map, all defined over $K$.  Suppose that $x\in X$ has the following properties:
\begin{enumerate}
\item every point in the orbit of $x$ under $\varphi$ avoids the indeterminacy loci of $\varphi$ and $f$;
\item $O_{\varphi}(x)$ is Zariski dense;
\item there is a finitely generated multiplicative subgroup $G$ of $K^*$ such that $f\circ \varphi^n(x)\in G$ for every $n\in \mathbb{N}_0$.
\end{enumerate}
Then there exist a dominant rational map $\Theta:X\dashrightarrow \mathbb{G}_m^d$ for some nonnegative integer $d$, and a dominant endomorphism $\Phi:\mathbb{G}_m^d\rightarrow \mathbb{G}_m^d$ such that the following diagram commutes
\begin{center}
    \begin{tikzpicture}[node distance=1.7cm, auto]
        \node (X1) {$X$};
        \node (X2) [right of=X1] {$X$};
        \node (Ad1) [below of=X1] {$\mathbb{G}_m^d$};
        \node (Ad2) [below of=X2] {$\mathbb{G}_m^d.$};

        \draw[->,dashed] (X1) to node {$\varphi$} (X2);
        \draw[->,dashed] (X1) to node [swap] {$\Theta$} (Ad1);
        \draw[->,dashed] (X2) to node {$\Theta$} (Ad2);
        \draw[->] (Ad1) to node [swap] {$\Phi$} (Ad2);
    \end{tikzpicture}
\end{center}
Moreover, $O_{\varphi}(x_0)$ avoids the indeterminacy locus of $\Theta$ and $f=g\circ \Theta$, where $g:\mathbb{G}_m^d \to \mathbb{G}_m$ is a map of the form
\[ g(t_1,\ldots,t_d) = Ct_1^{i_1}\cdots t_d^{i_d}\]
for some $i_1,\ldots,i_d\in \Z$ and some $C\in G$.
\end{thm}
In positive characteristic, the situation is more subtle and the conclusion to the statement of Theorem \ref{main theorem 2} fails (see Example \ref{exam}).
One can interpret this as saying that if the entire orbit of a point under a self-map has some ``coordinate'' that lies in a finitely generated multiplicative group then there must be a compelling geometric reason causing this to occur: in this case, it is that the dynamical behaviour of the orbit is in some sense
determined by the behaviour of a related system associated with a multiplicative torus.  In fact, we prove a more general version of this result involving semigroups of maps (see Corollary \ref{cor:semi}).
As a consequence of Theorem \ref{main theorem 2}, we get the following characterization of orbits whose values lie in a group of $S$-units, which shows that on arithmetic progressions they are well-behaved.
\begin{cor} Let $K$ be an algebraically closed field of characteristic zero and let $X$ be an irreducible quasiprojective variety with a dominant self-map $\varphi:X\to X$ and let $f:\mathbb{X}\to \mathbb{P}^1$ be a dominant rational map, all defined over $K$.  Suppose that $x\in X$ has the following properties:
\begin{enumerate}
\item every point in the orbit of $x$ under $\varphi$ avoids the indeterminacy loci of $\varphi$ and $f$;
\item there is a finitely generated multiplicative subgroup $G$ of $K^*$ such that $f\circ \varphi^n(x)\in G$ for every $n\in \mathbb{N}_0$.
\end{enumerate}
Then there are integers $p$ and $L$ with $p\ge 0$ and $L>0$ such that if $h_1,\ldots ,h_m$ generate $G$ then there are integer valued linear recurrences $b_{j,1}(n),\ldots ,b_{j,m}(n)$ for $j\in \{0,\ldots ,L\}$ such that
$$f\circ \varphi^{Ln+j} (x) = \prod_i h_i^{b_{j,i}(n)}$$ for $n\ge p$.
\label{orbit}
\end{cor}

Finally, we apply our results to $D$-finite power series, which, as stated above, are the generating functions of sequences that fall under the dynamical framework we study. We use Theorem~\ref{main theorem} to prove the following result.

\begin{thm}
\label{thm:Dfinite}
Let $F(x)=\sum_{n\geq 0} a_nx^n$ be a $D$-finite power series defined over a field $K$ of characteristic zero. Consider the sets
\[ N:= \{ n\geq 0 : a_n\in G\}\qquad and \qquad N_0:= \{ n\geq 0 : a_n\in G\cup \{0\}\},\]
where $G\leq K^*$ is a finitely generated group. Then $N$ and $N_0$ are both expressible as a union of finitely many infinite arithmetic progressions along with a set of Banach density zero.
\end{thm}
When $G=\{1\}$, this recovers a result of Methfessel \cite{Meth} and B\'ezivin \cite{Bez2}.  We conclude by revisiting the earlier works of P\'olya \cite{Pol} and B\'ezivin \cite{Bez} in terms of the dynamical results obtained in this paper.

\subsection{Conventions} Throughout, $\N:=\{1,2,3,\ldots\}$ and $\N_0:=\N\cup\{0\}$. If $R$ is a ring then $R^*$ is its multiplicative group of units.
An arithmetic progression is a set of the form $\{a+bn\}_{n\geq 0}\subseteq \N_0$ where $a,b\in \N_0$. A singleton counts as an arithmetic progression with $b=0$. A subset $N\subseteq \N_0$ is called \textit{eventually periodic} if it is a union of finitely many arithmetic progressions.

\subsection{Organization} In \textsection \ref{sec:AB} we develop a general theory of recurrences for sequences indexed by semigroups, which will be needed in proving our main results. In \textsection \ref{sec:MD} we prove a semigroup version of Theorem \ref{main theorem 2}. In \textsection \ref{sec:MT}, we give the proof of Theorem \ref{main theorem}, and in \textsection \ref{sec:Height} we give results on the heights of points in orbits for dynamical systems defined over $\bar{\Q}$.  Finally, \textsection \ref{sec:Dfin} gives applications of our results to $D$-finite power series.

\subsection*{Acknowledgments} We thank Dragos Ghioca for helpful comments. In this work, J.\ P.\ Bell and E.\ Hossain were supported by NSERC grant RGPIN-2016-03632;
S.\ Chen was supported by the NSFC
grants 11871067, 11688101 and the Fund of the Youth Innovation Promotion Association, CAS.

\smallskip

\section{Linear recurrences in abelian groups}\label{sec:AB}

In this section we develop the necessary background on general recurrences in abelian groups.  Because we will ultimately prove a result about a semigroup of morphisms, we will work with sequences  indexed by monoids in this section.  The proofs of these results become significantly simpler when the underlying monoid is just $(\N_0,+)$, which is the key case needed for dealing with a single map.

Let $(A,+)$ be an abelian group, let $S$ be a finitely generated monoid with identity $1=1_S$, and let $Z$ be a set upon which $S$ acts. Then the space of sequences $$A^{Z}:=\{u:Z\rightarrow A\}$$ is an abelian group, and for $u\in A^{Z}$ and $z\in Z$, we use the notations $u_z=u(z)$ interchangeably. Given a ring $R$, we let $R[S]$ denote the semigroup algebra of $S$ with coefficient ring $R$; that is, the elements of $R[S]$ are formal $R$-linear combinations of elements of $S$, where we multiply via the rule $(rs)\cdot (r's')=rr' ss'$ for $r,r'\in R$ and $s,s'\in S$ and we extend this multiplication bilinearly.  Notice that $A^{Z}$ has a natural $\Z[S]$-module structure given by the rule:
\[ (ms\cdot u)_z = m u_{sz} \qquad {\rm for~} m\in\Z, s\in S, {\rm and}~z\in Z. \]  We call the set of $f\in \Z[S]$ such that $f\cdot u=0$ the \emph{annihilator} of $u$.  It is not hard to check that the annihilator of $u$ is a two-sided ideal of $\Z[S]$.

\begin{defn}
\label{quasilinear defn}
Let $A$ be an additive abelian group, let $S$ be a finitely generated
monoid that acts on a set $Z$, and let $u\in A^{Z}$ be a sequence. Then $u$ satisfies an $S$-\emph{linear recurrence}
if the annihilating ideal $I$ of $u$ has the property that $\Z[S]/I$ is a finitely generated $\Z$-module.
We say that the sequence $u$ satisfies an $S$-\textit{quasilinear recurrence} if there are a set of generators $s_1,\ldots, s_d$ of $S$ and a natural number $M$ such that whenever $s_{i_1}\cdots s_{i_M}$ is an element of $S$ that is 
a product of $M$ elements of $s_1,\ldots ,s_d$, there is an element in $I$ of the form
\[\sum_{j=1}^M c_j s_{i_j}\cdots s_{i_M}\]
with $c_1,\ldots ,c_M\in \mathbb{Z}$ satisfying that $\gcd(c_1,\ldots ,c_M)=1$.
\end{defn}

The reason for introducing the notion of quasilinear recurrences is for later convenience, as it is often easier to demonstrate that a quasilinear recurrence exists.

\begin{example}
In general, a quasilinear recurrence may not be linear. To see this, let $S=\mathbb{N}_0$ and let $A$ be the additive group $(\Q,+)$.  Then if we consider the sequence $a_n=1/2^n$ and identify $\Z[S]$ with $\Z[x]$, then this sequence is annihilated by the primitive polynomial $f(x)=2x-1$, but it does not satisfy an $S$-linear recurrence since $a_{n+1}$ is never in the additive group generated by the initial terms $a_1,\ldots,a_n$.
\end{example}
We will make use of the following well-known facts throughout this section.
\begin{lemma} Let $T$ be a commutative noetherian integral domain and let $R$ be a finitely generated associate (but not necessarily commutative) $T$-algebra and suppose that $I$ and $J$ are two ideals of $R$ such that both $R/I$ and $R/J$ are finitely generated $T$-modules. Then the following hold:
\begin{enumerate}
\item[(a)] $R/IJ$ is also a finitely generated $T$-module;
\item[(b)] $I$ and $J$ are finitely generated as left ideals of $R$.
\end{enumerate}
\label{lem:FACTS}
\end{lemma}
\begin{proof}
We first prove (a).
Let $U=\{u_1,\ldots ,u_d\}$ be elements of $R$ with $u_1=1$ whose images span both $R/I$ and $R/J$ as $T$-modules and that generate $R$ as a $T$-algebra.
We prove that every finite product of elements from $u_1,\ldots ,u_d$ is congruent to a $T$-linear combination of elements of the form $u_iu_ju_ku_{\ell}$ modulo $IJ$.  (Since $u_1=1$, this includes products of smaller length.)  We prove this by induction on the length of the product, with the case for products of length at most four following by construction.  Suppose now that the result holds for all products of elements from $u_1,\ldots ,u_d$ of length less than $M$ with $M\ge 5$, and consider a product $u_{i_1}\cdots u_{i_M}$.  Then by our choice of $U$ we have
$$u_{i_1}\cdots u_{i_M-2} \equiv \sum a_{i} u_i ~(\bmod~I)$$ and
$$u_{i_{M-1}} u_{i_M} \equiv \sum b_{i} u_i ~(\bmod~J)$$ for $a_i,b_i\in T$.
Hence
$$\left(u_{i_1}\cdots u_{i_M-2} - \sum a_{i} u_i\right) \left( u_{i_{M-1}}u_{i_M}- \sum b_{i} u_i \right)\in IJ.$$
Then expanding the product, we see that
 $u_{i_1}\cdots u_{i_M}$ is congruent to a $T$-linear combination of products of $u_1,\ldots ,u_d$ of length at most $\max(M-1,3)=M-1$, and so by the induction hypothesis it is in the span of products of length at most $4$.  Thus (a) now follows by induction.

 For part (b), it suffices to prove that $I$ is finitely generated as a left ideal.  Then since $U$ spans $R/I$ as a $T$-module, there exist elements $c_{i,j,k}\in T$ such that
 $\alpha_{i,j}:=u_i u_j - \sum_k c_{i,j,k} u_k\in I$ for $1\le i,j,k\le d$.  Next, consider the submodule $M$ of $T^d$ consisting of $(t_1,\ldots, t_d)\in T^d$ such that $\sum t_i u_i\in I$.  Then since $T$ is noetherian, $M$ is finitely generated as a $T$-module and we pick elements
 $\beta_k = \sum t_{i,k} u_i$ for $k=1,\ldots ,s$ such that $(t_{1,k},\ldots ,t_{d,k})$ with $k=1,\ldots ,s$ generate $M$.  Then let $L$ denote the finitely generated left ideal in $R$ generated by
 the $\alpha_{i,j}$ and $\beta_k$.  By construction $L\subseteq I$ and so to complete the proof of (b) it suffices to show that $I\subseteq L$.  Since the $\alpha_{i,j}$ are in $L$, a straightforward induction gives that every finite product of $u_1,\ldots ,u_d$ is congruent modulo $L$ to a $T$-linear combination of $u_1,\ldots ,u_d$.  It follows that if $f\in I$ then
 $f \equiv \sum t_i u_i~(\bmod ~L)$ for some $t_1,\ldots ,t_d\in T$.  But since $L\subseteq I$, $t_1u_1+\cdots +t_du_d\in I$ and so by construction, $t_1u_1+\cdots +t_d u_d$ is a $T$-linear combination of the $\beta_k$ and hence it is in $L$.  It then follows that $f\in L$, giving us that $I\subseteq L$ and showing that $I=L$ and so $I$ is finitely generated as a left ideal.
\end{proof}
\begin{cor}
Let $S$ be a finitely generated monoid acting on a set $Z$, let $A$ and $B$ be abelian groups, and let $u\in A^Z$ and $v\in B^Z$ be sequences satisfying $S$-linear recurrences.
Then $(u,v)=(u_z,v_z)_{z\in Z}\in (A\oplus B)^Z$ also satisfies an $S$-linear recurrence. \label{rem:sumofgroups}
\end{cor}
\begin{proof}
Let $I$ and $J$ be respectively the annihilators of $u$ and $v$.   Then $\Z[S]/I$ and $\Z[S]/J$ are finitely generated $\Z$-modules and since $S$ is finitely generated, we have that
$\Z[S]/IJ$ is a finitely generated $\Z$-module.  Since $IJ$ annihilates both $u$ and $v$, it also annihilates $(u,v)$.  The result follows.
\end{proof}
The following lemma generalizes the classical Fatou's lemma on rational
power series in $\Z[[x]]$.
\begin{lemma}
\label{quasilinear implies linear}
Let $A$ be a finitely generated abelian group, let $S$ be a finitely generated monoid acting on a set $Z$, and let $u\in A^Z$. If $u=(u_z)_{z\in Z}$ satisfies an $S$-quasilinear recurrence then $u$ satisfies an $S$-linear recurrence.
\end{lemma}
\begin{proof}
Let $I$ denote the annihilator of $u$ in $\Z[S]$.
By Corollary \ref{rem:sumofgroups}, it suffices to prove this in the case when $A$ is a cyclic group.
We first consider the case when $A=\Z$.
Then quasilinearity gives that $\Z[S]/I \otimes_{\Z} \Q$ is finite-dimensional as a $\Q$-vector space, as there is some natural number $M$ such that it is spanned by the images of all words of a set of generators of length at most $M$.
We pick $t_1,\ldots ,t_d\in S$ such that their images span $\Z[S]/I \otimes_{\Z} \Q$ as a vector space and we let $R=\Z[S]/I$.
Consider the $\Z$-submodule $N$ of $\Z^d=A^d$ spanned by elements of the form $v_z:=(u(t_1\cdot z),\ldots ,u(t_d \cdot z))$ with $z\in Z$.  Then $N$ is finitely generated and hence there exist $z_1,\ldots ,z_m\in Z$ such that $N$ is generated by $v_{z_1},\ldots ,v_{z_m}$.

Then we define a homomorphism of additive abelian groups $\Psi: \Z[S]\to A^{m}$ given by
$$s\mapsto (u(s\cdot z_1),\ldots , u(s\cdot z_m)).$$  We claim that $f\in \Z[S]$ is in the kernel of $\Psi$ if and only if $f$ annihilates $u$.  It is clear that if $f$ annihilates $u$ then it is in the kernel of $\Psi$.
Conversely, suppose that $f$ is in the kernel of $\Psi$.
Then since the images of $t_1,\ldots ,t_d$ span $R\otimes_{\Z}\Q$, there are some positive integer $m$ and some integers $c_1,\ldots ,c_d$ such that $mf- c_1 t_1-\cdots -c_d t_d\in I$.
Then for $z\in Z$, $$mf\cdot u_z = (c_1 t_1+\cdots + c_d t_d)\cdot u_z =  \sum_{i=1}^d c_i u(t_i\cdot z).$$
Observe that the right-hand side is zero if $\sum_{i=1}^d c_i u(t_i z_j)=0$ for $j=1,\ldots ,m$.  But $$\sum_{i=1}^d c_i u(t_i z_j) = mf\cdot u_{z_j}=0$$ since $f$ is in the kernel of $\Psi$.  It follows that $m\cdot f$ annihilates $u$ and since $A$ is torsion-free we have $f$ is in $I$, giving us the claim.  It follows that $\Psi$ induces an injective map from $R$ into $A^m$, and so $R$ is a finitely generated abelian group, and so $u$ satisfies an $S$-linear recurrence.

Next suppose that $A=\Z/n\Z$ with $n>0$.  We suppose towards a contradiction that there exists a sequence $u\in A^Z$ that satisfies an $S$-quasilinear recurrence but not an $S$-linear recurrence.  We may also assume that $n$ is minimal among all positive integers for which there exists such a sequence in $(\Z/n\Z)^Z$.  We note that $n$ cannot be prime.  To see this, observe that there are generators $s_1,\ldots ,s_e$ of $S$ and some $M$ such that for every $M$-fold product $s_{i_1}\cdots s_{i_M}$ of elements from $s_1,\ldots ,s_e$ we have an element in $I$ of the form
$$\sum_{j=1}^M c_j s_{i_j}\cdots s_{i_M}$$ with $c_1,\ldots ,c_M\in \mathbb{Z}$ satisfying that $\gcd(c_1,\ldots ,c_M)=1$.  In particular, there is some smallest $j_0$ for which $n$ does not divide $c_{j_0}$.
If $n$ is a prime number, then $c_{j_0}$ is invertible modulo $n$.
Then by construction $$s_{i_1}\cdots s_{i_m} \equiv -c_{j_0}^{-1} \sum_{\ell=j_0+1}^M c_{\ell} s_{i_1}\cdots s_{i_{j_0-1}}s_{i_{\ell}}\cdots s_{i_M}~(\bmod ~I),$$ where
we take $c_{j_0}^{-1}$ to be an integer that is a multiplicative inverse of $c_{j_0}$ modulo $n$.  Then $\Z[S]/I$ is a $\Z/n\Z$-module spanned by words of length at most $M$ in this case. 
This contradicts our assumption, and so $n$ has a prime factor $p$ and $n=pn_0$ with $n_0>1$.  Now let $A_0=\{x\in A\colon px=0\}$. 
Then $\bar{u} := (u_z+A_0)_{z\in Z}$ satisfies an $S$-quasilinear recurrence and since $A/A_0$ is cyclic of order $n_0<n$, we have that $\bar{u}$ satisfies an $S$-linear recurrence by induction. Hence if $J$ denotes the annihilator of $\bar{u}$ then $\Z[S]/J$ is a finitely generated $\Z$-module.  Then for $f\in J$, we have $f\cdot u\in A_0^{Z}$ and satisfies an $S$-quasilinear recurrence and since $|A_0|=p$, it satisfies an $S$-linear recurrence by minimality of $n$.  In particular, for $f\in J$, if we let $J_f$ denote the annihilator of $f\cdot u$ then $\Z[S]/J_f$ is a finitely generated $\Z$-module.
Since $\Z[S]/J$ is a finitely generated $\Z$-module and $S$ is a finitely generated monoid, we have that $J$ is finitely generated as a left ideal.  We let $f_1,\ldots ,f_q$ denote a set of generators of $J$ as a left ideal.  Then by construction the ideal $J':=J_{f_1}f_1+\cdots + J_{f_q} f_q$ annihilates $u$. We claim that $\Z[S]/J'$ is a finitely generated $\Z$-module, from which it will follow that $u$ satisfies an $S$-linear recurrence. Since each $\Z[S]/J_{f_i}$ is a finitely generated $\Z$-module, $L:=J_{f_1} \cdots  J_{f_q}$ has the property that $\Z[S]/L$ is a finitely generated $\Z$-module.  By construction
$I\supseteq Lf_1+\cdots + Lf_q = LJ$ and since $\Z[S]/L$ and $\Z[S]/I$ are both finitely generated $\Z$-modules, so is $\Z[S]/LI$ and thus so is $\Z[S]/I$.  It now follows that $u$ satisfies an $S$-linear recurrence.
 \end{proof}

We require a few more basic facts about recurrences.
\begin{lemma}
Let $A$ be an abelian group, let $S$ be a finitely generated monoid, and let $u=(u_s)_{s\in S}$ be a sequence in $A^S$.  Suppose there is a 
surjective semigroup homomorphism $\Psi: S\to G$ where $G$ is a finite group and let $T$ be the semigroup $\Psi^{-1}(1)$.  
Then $T$ acts on the sets $Z_g:=\Psi^{-1}(g)$ for each $g\in G$.  Suppose that $T$ is finitely generated as a monoid and that for each $g\in G$ $u_g:=(u_{z})_{z\in Z_g}$ satisfies a $T$-linear recurrence.  Then $(u_s)$ satisfies an $S$-linear recurrence.
\end{lemma}
\begin{proof}  For $g\in G$, we let $I_g\subseteq \Z[T]$ denote the annihilator of $u_g$.  Then by assumption $\Z[T]/I_g$ is a finitely generated $\Z$-module and hence $\Z[T]/J$ is also a finitely generated $\Z$-module by Lemma \ref{lem:FACTS}, where $J:=\prod_{g\in G} I_g$.  By construction, if $f\in J$ then $f$ annihilates each $u_g$ and so, since $T$ acts on each $Z_g$, we have that $f$ annihilates $u$.  It follows that the ideal
$I:=\Z[S]J\Z[S]\subseteq \Z[S]$ is contained in the annihilator of $u$.  To finish the proof, it suffices to show that $\Z[S]/I$ is a finitely generated $\Z$-module.  We claim that there is 
a finite subset $U$ of $S$ such that every element of $S$ can be expressed in the form $u_1 t_1 u_2 t_2 \cdots u_{m-1} t_m u_{m}$, with $m\le |G|$.  To see this, we pick a set of generators $s_1,\ldots ,s_d$ of $S$ and let $U$ denote the set of elements of $S$ that can be expressed as a product of elements in $s_1,\ldots ,s_d$ of length at most $m$.  Then it is immediate that if $s$ is element of $S$, then $s$ has an expression of the form
$u_1 t_1 u_2 \cdots u_{p-1} t_{p-1} u_{p}$ for some $p$. For this element $s$, we pick such an expression with $p$ minimal.  If $p\le |G|$, there is nothing to prove, so we may assume that $p>|G|$.  Then $\Psi(u_1), \Psi(u_1u_2), \ldots , \Psi(u_1\cdots u_{p})$ are $p$ elements of $G$ and hence two of them must be the same.  So there exist $i, j$ with $1\le i<j\le p$ such that
$\Psi(u_1\cdots u_i) = \Psi(u_1\cdots u_j)$ and so $\Psi(u_{i+1}\cdots u_j)=1$.  In particular,
$t:=u_{i+1} t_{i+1} \cdots u_{j-1} t_{j-1} u_j\in T$ and thus we can rewrite $s$ as $u_1 t_1\cdots u_{i-1} (t_i t t_j) u_{j+1}\cdots t_{p-1} u_p$, which contradicts the minimality of $p$ in our expression for $s$.  The claim now follows.

Since $\Z[T]/J$ is a finitely generated $\Z$-module, there exists a finite subset $V$ of $T$ such that $\Z[T]/J$ is spanned by images of elements of $V$.  It follows that $\Z[S]/I$ is spanned as a $\Z$-module by images of elements of the form $u_1 t_1 u_2 t_2 \cdots u_{m-1} t_m u_{m}$ with $u_i\in U$ and $t_i\in V$ and $m\le |G|$.  Thus $\Z[S]/I$ is a finitely generated $\Z$-module and so $(u_s)$ satisfies an $S$-linear recurrence, as required.
\end{proof}
The following result is connected to the classical Skolem-Mahler-Lech theorem \cite{Lech}.
\begin{prop}
\label{recurrence zeros}
Let $A$ be a finitely generated abelian group, let $B\leq A$ be a subgroup, and let $(u_n)_{n\in \mathbb{N}_0}$
be an $A$-valued sequence that satisfies a $\mathbb{N}_0$-linear recurrence. Then $\{n\colon u_n\in B\}$ is eventually periodic.  \end{prop}

\begin{proof}
We may replace $(u_n)$ by the sequence $(u_n+B)$ taking values in the quotient group $A/B$ and thus we may assume without loss of generality that $B=(0)$.
Then $A/B$ is a finitely generated abelian group and hence is a direct sum of cyclic groups.
We write $A/B = C_1\oplus \cdots \oplus C_r$, with each $C_i$ cyclic.  Then $u_n = (u_{1,n},\ldots ,u_{r,n})$ with $(u_{i,n})$
a $C_i$-valued sequence satisfying an $\N_0$-linear recurrence for $i=1,\ldots ,r$.  Then $\{n\colon u_n=0\}$ is just the
intersection of the sets $$N_i:=\{n\colon u_{i,n}=0\}$$ for $i=1,\ldots ,r$.  Since a finite intersection of eventually periodic
subsets of $\mathbb{N}_0$ is again eventually periodic, we may assume without loss of generality that $A$ is cyclic and $B$ is $(0)$.
Now if $A=\mathbb{Z}$ then the result follows immediately from the Skolem-Mahler-Lech theorem \cite{Lech}.  On the other hand, if $A$ is
a finite cyclic group, then the result is immediate: if $u_n$ satisfies the recurrence $u_n+c_1 u_{n-1}+\cdots + c_d u_{n-d}=0$.  Then since $C$ is finite, there exist indices $p$ and $q$ with $d\le p<q$ such that $(u_{p-1},\ldots ,u_{p-d})=(u_{q-1},\ldots ,u_{q-d})$.  The recurrence then gives that $u_{n+p}=u_{n+q}$ for $n\ge 0$, and so $u_n$ is eventually periodic. The result follows. \end{proof}

We now turn our attention to multiplicative recurrences in $K^*$ with $K$ a field.

\begin{prop}
\label{recurrence units}
Let $K$ be a finitely generated extension of $\Q$, and let $(u_n)\in (K^*)^{\mathbb{N}_0}$ be a sequence satisfying a
multiplicative $\N_0$-quasilinear recurrence. Then in fact $(u_n)$ satisfies a (multiplicative) linear recurrence and
if $H$ is a finitely generated subgroup of $K^*$ then $\{n\colon u_n\in H\}$ is eventually periodic.
\end{prop}

This is not true without the hypothesis that $K$ is finitely generated as an extension of $\Q$.  For example if $K=\mathbb{C}$ and $u_n = \exp(2\pi i /2^n)$, then $u_n^2 = u_{n-1}$ and so $(u_n)$ satisfies a quasilinear recurrence.  But since $u_n$ is never in the subfield generated by $u_1,\ldots ,u_{n-1}$, $u_n$ does not satisfy a linear recurrence.
\begin{proof}[Proof of Proposition \ref{recurrence units}]
The assumption that $(u_n)$ is a quasilinear recurrence means that there are some $d\geq 0$ and integers $i_0,\ldots,i_d$ with $\gcd(i_0,\ldots,i_d)=1$ so that the following relation holds for all $n\geq 0$:
\[ u_n^{i_0}\cdots u_{n+d}^{i_d}=1. \]
Then if $G$ is the subgroup generated by $u_0,\ldots,u_d$, it follows that $u_n$ lies in the \textit{radical} of $G$ for all $n\geq 0$:
\[ \sqrt{G} := \{g\in K^* : g^m \in G \text{ for some } m\geq 1\}. \]
We will show that $\sqrt{G}$ is finitely generated, from which the desired conclusion follows from Lemma \ref{quasilinear implies linear}.

To see that $\sqrt{G}$ is finitely generated: since $K/\Q$ is finitely generated, $K$ is a finite extension of a function field $L=\Q(t_1,\ldots,t_m)$. Now let $R=\Z[G,t_1,\ldots,t_d]$ be the subring of $K$ generated by $t_1,\ldots,t_d$ and $G$, and let $F:=\mathrm{Frac}(R)$ so that $K/F$ is finite. Finally let $\cl{R}$ denote the integral closure of $R$ in $K$. Then $R$ is a finitely generated $\Z$-algebra, so the same is true for $\cl{R}$ \cite[Corollary 13.13]{Eisenbud}. It follows that the group of units $\cl{R}^{*}$ is finitely generated by Roquette's Theorem \cite{Roquette}. But $\sqrt{G}\leq \cl{R}^{*}$ since every element of $\sqrt{G}$ is integral over $R$.  Thus $\sqrt{G}$ is finitely generated.

Now let $H_0:=H\cap \sqrt{G}$.  Then by Proposition \ref{recurrence zeros} we have $\{n\colon u_n\in H_0\}$ is eventually periodic, and since $u_n\in H$ if and only if $u_n\in H_0$ we obtain the result.
\end{proof}

\smallskip

\section{Multiplicative dependence and $S$-unit equations}\label{sec:MD}

The goal of this section is to establish a key lemma which converts the statement of Theorem \ref{main theorem 2} into a problem about linear recurrences (in the sense of \textsection 2). Thus we may apply the results of \textsection 2 to obtain the main theorem.

\begin{defn}
\label{multiplicatively dependent}
Let $K$ be a field and let $k$ be a subfield of $K$. Elements $a_1,\ldots,a_n\in K^*$ are \textit{multiplicatively dependent modulo} $k^*$ if there are integers $i_1,\ldots,i_n\in \Z$, not all zero, such that $a_1^{i_1}\cdots a_n^{i_n}\in k^*$. If $a_1,\ldots ,a_n\in K^*$ are not multiplicatively dependent modulo $k^*$ then they are \emph{multiplicatively independent modulo} $k^*$.
\end{defn}
Observe that if $k$ is algebraically closed in $K$, the integers $i_0,\ldots,i_d$ in Definition \ref{multiplicatively dependent} can
be chosen to satisfy $\gcd(i_0,\ldots,i_d)=1$. Indeed, if $m=\gcd(i_0,\ldots,i_d)$, then $(i_0,\ldots,i_d)=(mj_0,\ldots,mj_d)$ for some $j_0,\ldots,j_d$, and set $g:=f_0^{j_0}\cdots f_d^{j_d}$. Then $g^m$ is in $k^*$. But then $g\in k^*$ as $k$ is algebraically closed. Since $\gcd(j_0,\ldots,j_d)=1$, this is the required multiplicative dependence modulo $k^*$.

Let $K$ be an algebraically closed field and let $X$ be an irreducible quasiprojective variety over $K$. For a group $G\leq K^{*}$ and rational functions $f_1,\ldots,f_n\in K(X)$, we set
\begin{equation} X_G(f_1,\ldots,f_n):=\{x\in X: f_1(x),\ldots,f_n(x)\in G \} = \bigcap_{i=1}^n f_i^{-1}(G).  \label{eq:XG}\end{equation}

Notice that if $X=\mathbb{A}^n$ and $f_i(x_1,\ldots,x_n):=x_i$ is a coordinate function, then $X_G(f_1,\ldots,f_n)$ is the set $X(G)$ of affine points with coordinates in $G$.  The set $X_G$ has been studied in \cite{BOSS} in the case $X=\mathbb{P}^1$ and $f_0=f_1=f:\mathbb{P}^1\rightarrow \mathbb{P}^1$ is a rational function; they determine exactly the form of such $f$ so that $X_G$ is infinite. Similarly, multiplicative dependence of values of rational functions has been studied in \cite{BOSS2}.

Now we state our key lemma.

\begin{lemma}
\label{dichotomy}
Let $K$ be an algebraically closed field of characteristic zero, let $G$ be a finitely generated multiplicative subgroup of $K^{*}$, let $X$ be an irreducible quasiprojective variety over $K$ of dimension $d$, and let $f_0,\ldots,f_d\in K(X)$ be $d+1$ rational functions on $X$. If $X_G(f_0,\ldots,f_d)$ is Zariski dense in $X$, then $f_0,\ldots,f_d$ are multiplicatively dependent modulo $K^*$.
\end{lemma}
\begin{proof}
Since the field extension $K(X)/K$ has transcendence degree $d$, the functions $f_0,\ldots,f_d$ must be algebraically dependent over $K$. Thus there is a polynomial relation
\[ \sum_{i_0,\ldots,i_d} c_{i_0\cdots i_d}f_0^{i_0}\cdots f_d^{i_d}=0 \]
where $c_{i_0\cdots i_d}\in K$ and the sum is over a finite set of indices in $\N_0^{d+1}$; this holds on some open subset of $X$. To simplify notation, let $I$ be the (finite) set of those indices $\alpha=(i_0,\ldots,i_d)\in \N_0^{d+1}$ where $c_{i_0\cdots i_d}$ is nonzero. For $\gamma=(i_0,\ldots,i_d)\in \Z^{d+1}$, we set \[ f^\gamma:=f_0^{i_0}\cdots f_d^{i_d}. \] Then for every $y\in X_G$, the $I$-tuple $(c_\alpha f^\alpha(y))_{\alpha\in I}$ is a solution to the $S$-unit equation
\[ \sum_{\alpha\in I} X_\alpha = 0 \]
in the group $\cl{G}$ generated by $G\cup \{c_\alpha:\alpha\in I\}$. This tuple may be \textit{degenerate} in the sense that some subsum vanishes, so we partition it into nondegenerate subtuples. Thus, for each partition $\pi\vdash I$, say $\pi=\{I_1,\ldots,I_m\}$, we let $X_{G,\pi}$ be the set of points $y\in X_G$ such that, for each $s=1,\ldots,m$, the $I_s$-tuple $(c_\alpha f^\alpha(y))_{\alpha\in I_s}$ is nondegenerate (\textit{i.e.} its sum vanishes, but no subsum vanishes). 
Note that there is a decomposition $X_G=\bigcup_{\pi\,\vdash I}X_{G,\pi}$ and hence there is some partition $\pi$ of $I$ such that
$X_{G,\pi}$ is Zariski dense in $X$.  Notice $X_{G,\pi}$ is empty if $\pi$ has some part of size $1$, and hence if $\pi=(I_1,\ldots ,I_m)$ then each $I_k$ has size at least two since $c_{\alpha}f^{\alpha}(y)\neq 0$ for $\alpha\in I$ and $y\in X_G$.
Thus there exist $\alpha,\beta$, two distinct indices in the same component $I_s$ of $\pi$.

By the Main Theorem on $S$-unit equations  \cite{ESS}, an $S$-unit equation in characteristic zero has only finitely many nondegenerate solutions up to scalar
multiplication \cite{ESS}. Let $(t_{1,\mu})_{\mu \in I_s}, \ldots, (t_{n,\mu})_{\mu \in I_s}$ be all solutions to the equation $$\sum_{\mu\in I_s} t_{\mu} = 0$$ up to scaling. Then for each $y\in X_{G,\pi}$, we know that $(c_\mu f^\mu(y))_{\mu\in I_s}$ is a multiple of some $(t_{j,\mu})_{\mu\in I_s}$,
so there is some $g\in \cl{G}$ such that
\[ c_\mu f^\mu(y) = g t_{j,\mu} \qquad \text{for all } \mu\in I_s. \]
Here $g,j$ may depend on $y$. But then for $\alpha,\beta\in I_s$, we can take quotients to get
\[ f^{\alpha-\beta}(y) = \frac{c_\beta t_{j,\alpha}}{c_\alpha t_{j,\beta}} \]
and there are only finitely many possible values for the right-hand side of this equation, independent of $y$.
Taking $\gamma=\alpha-\beta$, we have $f^{\gamma}(y)$ takes only finitely many values for $y \in X_{G,\pi}$.
Since $X$ is irreducible and $f^{\gamma}$ is constant on $X_{G,\pi}$, which is Zariski dense in $X$, we have $f^{\gamma}\in K^*$, which completes the proof.
\end{proof}

\subsection{Interpolation of $G$-valued orbits as recurrences.} Now we enter the setup of our semigroup-version of Theorem \ref{main theorem 2}.  We find it convenient to fix the following assumptions and notation for the remainder of this section.
\begin{notn}
We introduce the following notation:
\begin{enumerate}
\item we let $K$ be an algebraically closed field of characteristic zero;
\item we let $G$ be a finitely generated subgroup of $K^*$;
\item we let $X$ be an irreducible quasiprojective variety over $K$;
\item we let $\varphi_1,\ldots ,\varphi_m$ be rational self-maps of $X$ and we let $S$ denote the monoid generated by these maps under composition;
\item we let $S^{\rm op}$ denote the \emph{opposite semigroup}, which is, as a set, just $S$ but with multiplication $\star$ given by $\mu_1 \star \mu_2 = \mu_2\circ \mu_1$;
\item we let $f:X\to \mathbb{P}^1$ be a non-constant rational map;
\item we assume that $x_0\in X$ has the property that its forward orbit under $S$ is Zariski dense and each point avoids the indeterminacy loci of the maps $\varphi_1,\ldots ,\varphi_m$ and $f$.
\end{enumerate}
\label{notn1}
\end{notn}

With these data fixed, we may thus define a sequence $u$ in $ K^S$ by
\[ u_{\varphi}:=f(\varphi(x_0)). \]
Notice that the semigroup algebra $\mathbb{Z}[S^{\rm op}]$ acts on $K^S$ via the rule $$\varphi\cdot (v_{\mu})_{\mu\in S} = (v_{\mu\circ \varphi})_{\mu\in S}$$ for $\varphi\in S$.
In this section, we analyze the case when $u_{\varphi}\in G$ for every $\varphi\in S$.

\begin{prop}
\label{multiplicative linear recurrence}
Adopt that assumptions and notation of Notation \ref{notn1}. If $f(\varphi(x_0))\in G$ for every $\varphi\in S$ then $(f\circ\varphi(x_0))_{\varphi\in S^{\rm op}}$ satisfies a multiplicative $S^{\rm op}$-linear recurrence.
\end{prop}

\begin{proof} We let $C\subseteq (K^*)^S$ denote the set of constant sequences and let $v$ denote the image of $u$ in $K^S/C$.  We first show that $v$ satisfies an $S^{\rm op}$-quasilinear recurrence.  Let $d$ denote the dimension of $X$ and
let $\varphi_{i_1}\circ \cdots \circ \varphi_{i_{d+1}}$ be a $(d+1)$-fold composition of $\varphi_1,\ldots ,\varphi_m$, let
$\mu_j:=\varphi_{i_j}\circ \cdots \circ \varphi_{i_{d+1}}$ for $j=1,2,\ldots ,d+1$, and let $f_j=f\circ\mu_j$.  Then we take $X_G:=X_G(f_1,\ldots,f_{d+1})$, as in Equation (\ref{eq:XG}). The assumption that $u_{\varphi}\in G$ for $\varphi\in S$ implies that $X_G$ contains the orbit of $x_0$ under $S$, which is dense. Thus $\cl{X_G}=X$ and it follows from Lemma \ref{dichotomy} that $f_1,\ldots,f_{d+1}$ are multiplicatively dependent modulo $K^*$.  Hence
\[ f_1^{p_1}\cdots f_{d+1}^{p_{d+1}}\equiv c \]
where $c\in K^*$ is a constant and $p_1,\ldots,p_{d+1}\in \Z$ with $\gcd(p_1,\ldots ,p_{d+1})=1$.
Now evaluating this at $\varphi(x_0)$ gives
\[  u_{\varphi\star \mu_1}^{p_1}\cdots u_{\varphi\star \mu_{d+1}}^{p_{d+1}}=c \quad \text{for all $\varphi\in S$}. \]
In particular, $p_1\mu_1+\cdots + p_{d+1} \mu_{d+1}\in \Z[S^{\rm op}]$ annihilates $v$.  It follows that $v$ satisfies an $S^{\rm op}$-quasilinear recurrence and it now follows from Lemma \ref{quasilinear implies linear} that it satisfies an $S^{\rm op}$-linear recurrence.
We now claim that $u$ satisfies an $S^{\rm op}$-linear recurrence.  To see this, let $I\subseteq \Z[S^{\rm op}]$ denote the annihilator of $v$.  Then we have shown that $R:=\Z[S^{\rm op}]/I$ is a finitely generated $\Z$-module.  In particular, there exists some $M$ such that $R$ is spanned as a $\Z$-module by compositions of $\varphi_1,\ldots ,\varphi_m$ of length at most $M$.  Now let $J$ denote the annihilator of $u$.  We claim that $ \Z[S^{\rm op}]/J$ is spanned as a $\Z$-module by compositions of length at most $M+1$, which will complete the proof that $u$ satisfies an $S^{\rm op}$-linear recurrence.
So to show this, let $\varphi$ be a composition of length $\ell\ge M+1$.  We shall show by induction on $\ell$ that $\varphi$ is equivalent mod $J$ to a $\Z$-linear combination of compositions of $\varphi_1,\ldots ,\varphi_m$ of length $M+1$, with the base case $\ell=M+1$ being immediate.  So suppose that the claim holds whenever $\ell<q$ and consider the case when $\ell=q$.  Then we can write $\varphi = \mu\circ \varphi_j$ for some $\mu$ that is a composition of length $q-1$ and some $j\in \{1,\ldots ,m\}$.
Then $\mu \equiv \sum m_j \mu_j~(\bmod~I)$, where the $m_j$ are integers and the $\mu_j$ are all compositions of $\varphi_1,\ldots ,\varphi_m$ of length at most $M$.
In particular, $(\mu-\sum m_j\mu_j)\cdot u \in C$ and so $(\varphi_j-1)\star (\mu-\sum m_j\mu_j)\cdot u = 0$.  Hence
$$\mu\circ \varphi_j \equiv \mu - \sum_j m_j (\mu_j\circ \varphi_j - \mu_j)~(\bmod ~J).$$
By the induction hypothesis the right-hand side is equivalent mod $J$ to a $\Z$-linear combination of compositions of $\varphi_1,\ldots ,\varphi_m$ of length $M+1$, and so we now obtain the result.
\end{proof}

Proposition \ref{multiplicative linear recurrence} gives a combinatorial description of the sequence $ (u_\phi)_{\phi \in S}$. We now give a more geometric interpretation of this result.

\begin{cor}
\label{cor:semi}
Adopt the assumptions and notation of Notation \ref{notn1}.  If $f\circ \varphi(x_0)\in G$ for every $\varphi$ in the monoid $S$ generated by $\varphi_1,\ldots ,\varphi_m$ then there exist a dominant rational map $\Theta:X\dashrightarrow \mathbb{G}_m^d$ that is defined at each point in $O_{\varphi}(x_0)$ and endomorphisms $\Phi_1,\ldots ,\Phi_m:\mathbb{G}_m^d\rightarrow \mathbb{G}_m^d$ such that the following diagram commutes
\begin{center}
    \begin{tikzpicture}[node distance=2.7cm, auto]
        \node (X1) at (0,0) {$X$};
        \node (X2) at (3,0) {$X$};
        \node (Ad1) at (0,-2) {$\mathbb{G}_m^d$};
        \node (Ad2) at (3,-2) {$\mathbb{G}_m^d.$};

        \draw[->,dashed] (X1) to node  {$\varphi_1,\ldots ,\varphi_m$} (X2);
        \draw[->,dashed] (X1) to node [swap]{$\Theta$} (Ad1);
        \draw[->,dashed] (X2) to node {$\Theta$} (Ad2);
        \draw[->] (Ad1) to node [swap] {$\Phi_1,\ldots ,\Phi_m$} (Ad2);
    \end{tikzpicture}
\end{center}

Moreover, $f=g\circ \Theta$, where $g:\mathbb{G}_m^d \to \mathbb{G}_m$ is a map of the form
\[ g(t_1,\ldots,t_d) = Ct_1^{i_1}\cdots t_d^{i_d}\]
for some $i_1,\ldots,i_d\in \Z$.
\end{cor}

\begin{proof}
We let $S^{\rm op}$ denote the opposite monoid of $S$.  Then by Proposition \ref{multiplicative linear recurrence} the sequence $u:=(f\circ \varphi(x_0))_{\varphi\in S}\in G^{S}$ satisfies a multiplicative $S^{\rm op}$-linear recurrence.  
It follows that there is some $M$ such that every $M$-fold composition of $\varphi_1,\ldots ,\varphi_m$ is congruent, modulo the annihilator of $u$, to a $\Z$-linear combination of $j$-fold compositions of these endomorphisms, as $j$ ranges over numbers $<M$.
Let $W$ denote the set of $j$-fold compositions of $\varphi_1,\ldots ,\varphi_m$ with $j<M$.
Then we construct a rational map $\Theta:X\dashrightarrow \mathbb{G}_m^L$, where $L=|W|$, given by
$\Theta(x) = (f\circ \varphi(x))_{\varphi\in W}$.  Now let $i\in \{1,\ldots ,m\}$ and consider
$\Theta(\varphi_i(x))=(f\circ \varphi\circ \varphi_i(x))_{\varphi\in W}$.  By construction $f=\pi\circ \Theta$, where $\pi$ is a suitable projection.

Then for $\varphi\in W$ and $i\in \{1,\ldots ,m\}$, $ \varphi\circ \varphi_i$ either remains in $W$ or it is an $M$-fold composition of $\varphi_1,\ldots ,\varphi_m$, in which case the fact that $u$ satisfies an $S^{\rm op}$-linear recurrence gives that
there exist integers $p_{\mu}$ for each $\mu\in W$ such that $$f\circ \varphi\circ \varphi_i(x) = \prod_{\mu\in W} (f\circ \mu(x))^{p_w}$$ for all $x$ in the $S$-orbit of $x_0$.  In particular, since the $S$-orbit of $x_0$ is Zariski dense in $X$,
$\Theta\circ \varphi_i= \Psi_i\circ \Theta$ for some self-map $\Psi_i$ of $\mathbb{G}_m^L$ of the form $$(u_1,\ldots ,u_L)\mapsto \left(\prod_j u_j^{p_{1,j}},\ldots ,\prod_j u_j^{p_{L,j}}\right).$$  In particular, each $\Psi_i$ is a group endomorphism of the multiplicative torus.
Now let $Y$ denote the Zariski closure of the $S$-orbit of $x_0$ under $\Theta$.  Then by construction $Y$ has a Zariski dense set of points in $G^L\le \mathbb{G}_m^L$ and is irreducible.  Then a theorem of Laurent \cite[Th\'eor\`eme 2]{Laurent} gives that $Y$ is a translate of a subtorus of $\mathbb{G}_m^L$.  In particular, $Y\cong \mathbb{G}_m^d$ for some $d\le {\rm dim}(X)$ and $\Psi_1,\ldots ,\Psi_d$ restrict to endomorphisms of $Y$.   Moreover, since $Y$ is a translation of a subtorus, the restriction of $\pi$ to $Y$ induces a map $g:\mathbb{G}_m^d\to \mathbb{G}_m$ of the form $g(u_1,\ldots ,u_d)\mapsto Cu_1^{q_1}\cdots u_d^{q_d}$.  The result now follows.
\end{proof}
\begin{remark}\label{rem:semi}
In fact, it can be observed that $\Theta(x_0)\in G^d$ and that $\Psi_i$ induce maps of $\mathbb{G}_m^d$ of the form $$(x_1,\ldots ,x_d)\mapsto (\lambda_1 x_1^{p_{1,1}}\cdots x_d^{p_{1,d}}, \ldots ,
\lambda_d x_1^{p_{d,1}}\cdots x_d^{p_{d,d}})$$ with $\lambda_1,\ldots ,\lambda_d\in G$; finally, $g(x_1,\ldots ,x_d)\mapsto Cx_1^{q_1}\cdots x_d^{q_d}$ with $C\in G$.
\end{remark}
 The following example shows that the conclusion to Corollary \ref{cor:semi} does not necessarily hold if $K$ has positive characteristic.
\begin{example}
Let $K=\bar{\mathbb{F}}_p(u)$ and let $X=\mathbb{P}^1_K$.  Then we have a map $\varphi:X\to X$ given by $t\mapsto t^p+1$
and let $f:X\to \mathbb{P}^1$ be the map $f(t)=t$.  Notice that if we take $x_0=u$ then $f\circ \varphi^n(u) = u^{p^n}+n = u^{p^n} (1+n/u)^{p^n}$ and hence $\varphi^n(u)$ lies in the finitely generated subgroup $G$ of $K^*$ generated by
$u$ and $1+n/u$ for $n=1,2,\ldots ,p-1$.  Then if the conclusion to Corollary \ref{cor:semi} held, we would necessarily have $d=1$ since $\Theta$ is dominant and $f\circ \varphi^n(u)$ has infinite orbit.  Thus the function fields of $\mathbb{P}^1$ and $\mathbb{G}_m^d$ are both isomorphic to $K(t)$ and the commutative diagram given in the statement of Corollary \ref{cor:semi} would give rise to a corresponding diagram at the level of functions fields:
\begin{center}
    \begin{tikzpicture}[node distance=1.7cm, auto]
        \node (X1) {$K(t)$};
        \node (X2) [right of=X1] {$K(t)$};
        \node (Ad1) [below of=X1] {$K(t)$};
        \node (Ad2) [below of=X2] {$K(t)$};

        \draw[->,dashed] (X1) to node {$\Phi^*$} (X2);
        \draw[->,dashed] (X1) to node [swap] {$\Theta^*$} (Ad1);
        \draw[->,dashed] (X2) to node {$\Theta^*$} (Ad2);
        \draw[->,dashed] (Ad1) to node [swap] {$\varphi^*$} (Ad2);
    \end{tikzpicture}
\end{center}
with $\varphi^*(t)=t^p+1$ and $\Phi^*(t) = C t^a$ for some integer $a$ and some $C\in K^*$.  Moreover, $f^*=\Theta^*\circ g^*$ and since $f^*$ is the identity map of $K(t)$, $\Theta^*$ and $g^*$ are automorphisms of $K(t)$; since $g^*(t)=C' t^b$ for some integer $a$ and some $C'\in K^*$, we have $b=\pm 1$, and so $\Theta^*(t) =C'^{-b} t^b$.
But now $\Theta^*\circ \Phi^*(t) = (C')^{-ab} C t^{ab}$ and $\varphi^*\circ \Theta^*(t) = (C')^{-b} (t^p+1)^{-b}$, and so the two sides do not agree.
 \label{exam}
\end{example}

\begin{proof}[Proof of Corollary \ref{orbit}]
 For each $n\ge 1$, we let $X_{\ge n}$ denote the Zariski closure of $\{ \varphi^m(x_0)\colon m\ge n\}$.  Since the $X_{\ge i}$ form a descending chain of closed sets and since $X$ endowed with the Zariski topology is a 
 noetherian topological space, there is some $m$ such that $X_{\ge m}=X_{\ge m+1}=\cdots $.  We let $Y=X_{\ge m}$ and we let $Z_1,\ldots ,Z_r$ denote the irreducible components of $Y$.  By our choice of $m$, $\varphi$ induces a dominant rational self-map of $Y$ and in particular there is 
some permutation $\sigma$ of $\{1,\ldots ,r\}$ such that
$\varphi(Z_i)$ is Zariski dense in $Z_{\sigma(i)}$.  It follows that there is some $L$ such that $\varphi^L$ maps each $Z_i$ into itself. Let $j\in \{m,\ldots ,m+L-1\}$.  Then $\varphi^{j}(x_0)\in Z_i$ for some $i$.
Then by the above, we have $\{\varphi^{Ln+j}(x_0)\colon n\ge 0\}$ is Zariski dense in $Z_i$.
Moreover, $f(\varphi^{Ln+j}(x_0))\in G$ for every $n\ge 0$ and so there are some $e\ge 0$ and some endomorphism $\Psi:\mathbb{G}_m^e\to \mathbb{G}_m^e$ and a map $g:\mathbb{G}_m^e \to \mathbb{G}_m$ such that $f(\varphi^{Ln+j}(x_0))= g\circ \Psi^n (z_0)$ for some $z_0\in \mathbb{G}_m^e$ whose coordinates lie in $G$.  Let $h_1,\ldots ,h_m$ be a set of generators for $G$.
Then $$\Psi^n(z_0) = (h_1^{a_{1,1}(n)}\cdots h_m^{a_{1,m}(n)}, \ldots , h_1^{a_{e,1}(n)}\cdots h_m^{a_{e,m}(n)})$$ for some integer-valued sequence $a_{i,j}(n)$.  (There may be several choices for the sequences $a_{i,j}(n)$ if the $h_i$ are not multiplicatively independent.)
Since $$\Psi(x_1,\ldots ,x_e) = (h_1^{p_{1,1}}\cdots h_m^{p_{1,m}} x_1^{q_{1,1}}\cdots x_e^{q_{1,e}},\ldots , h_1^{p_{e,1}}\cdots h_m^{p_{e,m}} x_1^{q_{e,1}}\cdots x_e^{q_{e,e}}),$$
there is a choice of sequences $a_{i,j}(n)$ such that there are an integer matrix $A$ and an integer vector ${\bf p}$ such that
$${\bf v}(n+1)=A{\bf v}(n) + {\bf p}$$ for every $n\ge 0$, where ${\bf v}(n)$ is the column vector whose entries are $a_{i,j}(n)$ for $i=1,\ldots ,e$, and $j=1,\ldots ,m$ in some fixed ordering of the indices that does not vary with $n$.
In particular, if $Q(x)=q_0+q_1x+\cdots + q_r x^r \in \mathbb{Z}[x]$ then
$$Q(A){\bf v}(n) = q_0 {\bf v}(n) + q_1{\bf v}(n+1)+\cdots + q_r {\bf v}(n+r) + {\bf b}_Q$$ for $n\ge 0$, where ${\bf b}_Q$ is an integer vector that depends upon $Q$ but not upon $n$.
In particular, if we take $Q(x)$ to be the characteristic polynomial of $A$, the Cayley-Hamilton theorem gives that the vectors ${\bf v}(n)$ satisfy a non-trivial affine linear recurrence of the form
$$0 = q_0 {\bf v}(n) + q_1{\bf v}(n+1)+\cdots + q_r {\bf v}(n+r) + {\bf b}_Q$$ for $n\ge 0$.
In particular, substituting $n+1$ for $n$ into this equation and subtracting from our original equation gives a recurrence
$$0 = q_0 {\bf v}(n) + (q_1-q_0){\bf v}(n+1)+\cdots + (q_r -q_{r-1}) {\bf v}(n+r)  - q_r {\bf v}(n+r+1).$$
It follows that each $a_{i,j}(n)$ satisfies a linear recurrence.  Then applying the map $g$ and using the fact that a sum of sequences satisfying a linear recurrence also satisfies a linear recurrence now gives the result.
\end{proof}

\section{Proof of Theorem \ref{main theorem}}\label{sec:MT}

In this section we prove Theorem \ref{main theorem}. The set up is as follows: $X$ is a quasiprojective variety defined over a field $K$ of characteristic zero, $\varphi:X\dashrightarrow X$ is a rational map, $x_0\in X$ is a point whose forward $\varphi$-orbit is well-defined, $f:X\dashrightarrow K$ is a rational function defined on this orbit, and $G$ is a finitely generated subgroup of $K^*$. Finally, we let
\[  N:= \{n\in \N_0 : f(\varphi^n(x_0))\in G\}. \]

We first show that if $N$ has a positive Banach density then it must contain an infinite arithmetic progression. Once a single arithmetic progression is obtained, we then use noetherian induction to show that $N$ is a union of finitely
many arithmetic progressions together with a set of Banach density zero.

\subsection{A single arithmetic progression.} With notation as above, in this section we will prove:

\begin{prop}
\label{arithmetic progression}
Let $X$ be a quasiprojective variety over an algebraically closed field $K$ of characteristic zero, let $\varphi:X\dashrightarrow X$ be a rational map, let $f:X\dashrightarrow K$ be a rational function, and let $G\leq K^*$ be a finitely generated group. Suppose that $x_0\in X$ is a point with well-defined forward $\varphi$-orbit that also avoids the indeterminacy locus of $f$.  Then if the set
\[ N:=\left\{n\in \N_0 : f(\varphi^n(x_0)) \in G\right\} \]
has a positive Banach density then it contains an infinite arithmetic progression.
\end{prop}

To prove this result, we require a lemma.
\begin{lemma} Let $K$ be an algebraically closed field, let $X$ be an irreducible quasiprojective variety over $K$, let $\varphi:X\to X$ and $f:X\to \mathbb{P}^1$ be rational maps, and let $x_0$ be a point whose forward orbit under $\varphi$ is defined and is Zariski dense and avoids the indeterminacy locus of $f$.  If the set of $n$ for which $f\circ \varphi^n(x_0)=0$ 
has Banach density zero and if $u_n:=f\circ \varphi^n(x_0)$ has the property that there exist $C\neq 0$ and integers $i_0,\ldots ,i_d$ with $i_0i_d\neq 0$ and $\gcd(i_0,\ldots ,i_d)=1$ such that
$u_n^{i_0}\cdots u_{n+d}^{i_d} = C$ for every $n\ge 0$ then:
\begin{enumerate}
\item $u_n\in K^*$ for all $n$;
\item the set of $n$ for which $u_n\in G$ is an eventually periodic set, whenever $G$ is a finitely generated subgroup of $K^*$.
\end{enumerate}
\label{lem:un}
\end{lemma}
\begin{proof}
Since $u_n^{i_0}\cdots u_{n+d}^{i_d} = C$ and since $x_0$ has a Zariski dense orbit, we have
$f^{i_0} = C\prod_{j=1}^d (f\circ \varphi^j)^{-i_j}$.  In particular, if $f$ has a zero at $\varphi^n(x_0)$ for some $n$, then there is some $j\in \{1,2,\ldots ,d\}$ for which $f\circ \varphi^j$ has a zero or a pole at
$\varphi^n(x_0)$.  But since the orbit of $x_0$ under $\varphi$ avoids the indeterminacy locus of $f$, $f(\varphi^{j+n}(x_0))=0$ for some $j\in \{1,2,\ldots ,d\}$. Hence if $u_n=0$ then $u_{n+j}=0$ for some $j\in \{1,2,\ldots ,d\}$.
In particular, $\{n\colon u_n=0\}$ has a positive Banach density, a contradiction.  Thus $u_n\in K^*$.  In fact, there is a finitely generated extension of $\Q$, $K_0\subseteq K$, such that $x_0\in X(K_0)$ and such that $\varphi$ and $f$ are defined over $K_0$.  It follows that $u_n\in K_0^*$ for all $n$ and using the equation $u_n^{i_0}\cdots u_{n+d}^{i_d} = C$ and substituting $n+1$ for $n$ and taking quotients, we have
$$u_{n+d+1}^{i_{d}} u_{n+d}^{i_{d-1}-i_d} \cdots u_{n+1}^{i_0-i_1} u_n^{-i_0}=1.$$
Moreover, it is straightforward to show that $\gcd(i_0,i_0-i_1,\ldots ,i_{d-1}-i_d,i_d)=1$ and so $(u_n)$ satisfies an $\mathbb{N}_0$-quasilinear recurrence.  But that means it satisfies a linear recurrence by Lemma \ref{quasilinear implies linear}.  In particular, the $u_i$ are all contained in a subfield $K_0$ of $K$ that is finitely generated over $\Q$
and so the result follows from Proposition \ref{recurrence units}.
\end{proof}

\begin{proof}[Proof of Proposition \ref{arithmetic progression}]
By \cite[Theorem 1.4]{BGT} there is some positive integer $L$ such that for $j\in \{0,\ldots ,L-1\}$ we have $\mathcal{Z}_j:=\{n\colon f\circ \varphi^{Ln+j}(x_0)=0\}$ is a either 
a set of Banach density zero or contains all sufficiently large natural numbers.  If $\delta(N)>0$ then there is some $j$ such that $N\cap (L\mathbb{N}_0+j)$
has a positive Banach density and such that $\mathcal{Z}_j$ has Banach density zero.
Then we can replace $\varphi$ by $\varphi^L$ and $x_0$ by $\varphi^j(x_0)$ and we may assume without loss of generality that the set of $n$ for
which $f\circ \varphi^n(x_0)=0$ has Banach density zero.

Let $\mathcal{S}$ denote the collection of Zariski closed subsets $Y$ of $X$ for which there exists a rational self-map $\Psi: Y\dashrightarrow Y$ and $y_0\in Y$ whose
forward orbit under $\Psi$ is well-defined and avoids the indeterminacy locus of $f$ and such that the following hold:
\begin{enumerate}
\item[(i)] $N(Y,y_0,\Psi,f,G):=\{n\colon f\circ \Psi^n(y_0)\in G\}$ has a positive Banach density but does not contain an infinite arithmetic progression;
\item[(ii)] $\{n\colon f\circ \Psi^n(y_0)=0\}$ has Banach density zero.
\end{enumerate}
If $\mathcal{S}$ is empty, then we are done.  Thus we may assume $\mathcal{S}$ is non-empty and since $X$ is a noetherian topological space, there is some minimal element $Y$ in $\mathcal{S}$.  By assumption, there exists a rational self-map $\Psi: Y\dashrightarrow Y$ and $y_0\in Y$ such that conditions (i) and (ii) hold.

Observe that the orbit of $y_0$ under $\Psi$ must be Zariski dense in $Y$, since otherwise, we could replace $Y$ with the Zariski closure of this orbit and construct a smaller counterexample.  We also note that $Y$ is necessarily irreducible.  To see this, suppose towards a contradiction, that this is not the case and let $Y_1,\ldots ,Y_r$ denote the irreducible components of $Y$, with $r\ge 2$. Then since the orbit of $y_0$ is Zariski dense, $\Psi$ is dominant and hence it permutes the irreducible components of 
$Y$ in the sense that there is a permutation $\sigma$ of $\{1,\ldots ,r\}$ with the property that $\Psi(Y_i)$ is Zariski dense in $Y_{\sigma(i)}$.  It follows that there is some $M>1$ such that $\Psi^M$ maps $Y_i$ into $Y_i$ for every $i$.
Now there must be some $j\in \{0,\ldots ,M-1\}$ such that $(M\mathbb{N}+j)\cap N_Y$ has a positive Banach density.  Then $\Psi^j(y_0)\in Y_i$ for some $i$, and so by construction
$N(Y_i,\Psi^j(y_0), \Psi^M, f, G)$ has a positive Banach density.  Since $Y_i$ is a proper closed subset of $Y$, by minimality of $Y$, $N(Y_i,\Psi^j(y_0), \Psi^M, f, G)$ contains an infinite arithmetic progression.  But $N(Y_i,\Psi^j(y_0), \Psi^L, f, G)\subseteq N(Y,y_0,\Psi,f,G)$ and so $N(Y,y_0,\Psi,f,G)$ contains an infinite arithmetic progression, a contradiction.  Thus $Y$ is irreducible.

Let $d:=\dim(Y)$. Since the Banach density of $N(Y,y_0,\Psi,f,G)$ is positive, a version of Szemer\'{e}di's Theorem \cite{Sze} due to Furstenberg \cite[Theorem 1.4]{Fur} gives that there is a set $A$ of 
positive Banach density and a fixed integer $b\geq 1$ such that $N$ contains the finite progression
\[ a,a+b,a+2b,\ldots,a+db \]
for every $a\in A$. Setting, $f_n:=f\circ\Psi^{bn}$ for $n\geq 0$, we have defined $d+1$ rational functions $f_0,\ldots,f_d$, so by Lemma \ref{dichotomy} either the set \[ Y_G:=Y_G(f_0,\ldots,f_d)=\{x\in Y : f_0(x),\ldots,f_d(x)\in G\} \] is contained in a proper subvariety of $Y$, or the functions $f_0,\ldots,f_d$ satisfy some multiplicative dependence relation. We proceed by ruling out the first possibility.
Suppose that $\cl{Y_G}\subsetneq Y$. Since $\Psi^a(y_0)\in \cl{Y_G}$ for every $a\in A$, the set
\[ P:=\{n\in \N_0: \Psi^n(y_0)\in \cl{Y_G}\} \]
has a positive Banach density. Hence \cite[Theorem 1.4]{BGT} gives that $P$ is a union of infinite arithmetic progressions $A_1,\ldots,A_r$ and a set of density zero.  In particular, since $P$
has a positive Banach density, $P$ contains an infinite arithmetic progression.  But since $P\subseteq N(Y,y_0,\Psi,f,G)$, we then see $N(Y,y_0,\Psi,f,G)$ contains an infinite arithmetic progression, a contradiction.  It follows that $Y_G$ is Zariski dense in $Y$.

Combining this with Lemma \ref{dichotomy}, we conclude that there is a multiplicative dependence relation
\begin{equation}
\prod_{s=0}^{d} f(\Psi^{sb}(x))^{i_s}=C\in K^*,
\label{eq:XXX}
\end{equation}
where $i_0,\ldots,i_{d}\in \Z$  with $\gcd(i_0,\ldots,i_{d})=1$.
Then for $a\in \{0,\ldots ,b-1\}$ we let $u_a(n):=f(\Psi^{a+bn}(y_0))$.  Evaluating Equation (\ref{eq:XXX}) at $x=\Psi^{a+bn}(y_0)$ then gives the relation
\[ u_a(n)^{i_0}\cdots u_a(n+d)^{i_{d}} = C. \] 
and so Lemma \ref{lem:un} gives that $u_a(n)\in K^*$ for all $n\ge 0$ and that $u_a(n)$ satisfies a multiplicative $\N_0$ linear recurrence and that the set of $n$ for which $u_a(n)\in G$ is eventually periodic.  In particular, 
since there is some $a$
for which the set $\{n\colon u_a(n)\in G\}$ has a positive Banach density, for this $a$, $\{n\colon u_a(n)\in G\}$ contains an infinite arithmetic progression $c+e\mathbb{N}_0$.  This then gives that $N$ contains the infinite arithmetic progression $a+b(c+e\mathbb{N}_0)= (a+bc)+be\mathbb{N}_0$, as required.
\end{proof}

\subsection{A union of arithmetic progressions.} We now use Proposition \ref{arithmetic progression} to prove Theorem \ref{main theorem}.
\begin{proof}[Proof of Theorem \ref{main theorem}]
First, by \cite[Theorem 1.4]{BGT} there is some positive integer $L$ such that for $j\in \{0,\ldots ,L-1\}$ we have $\mathcal{Z}_j:=\{n\colon f\circ \varphi^{Ln+j}(x_0)=0\}$ is either a set of Banach density zero or contains all sufficiently large natural numbers.  Then to prove the result, it suffices to prove that for every natural number $j\in \{0,\ldots ,L-1\}$, the set of $n$ for which $f\circ \varphi^{Ln+j}(x_0)\in G$ is a finite union of arithmetic progressions along with 
a set of Banach density zero.  In the case that $\mathcal{Z}_j$ contains all sufficiently large natural numbers, this is immediate; hence we may replace $\varphi$ by $\varphi^L$ and $x_0$ by some point in the orbit under $\varphi$ and assume without loss of generality that the set $\mathcal{Z}$ of $n$ for which $f\circ \varphi^n(x_0)=0$ 
has Banach density zero.  We now let $X_{\ge i}$ denote the Zariski closure of $\{\varphi^n(x_0)\colon n\ge i\}$.  Then as in the proof of Corollary \ref{orbit}, we have that there is some $m$ such that $X_{\ge m}=X_{\ge m+1}=\cdots $ and without loss of generality we may replace $X$ with $X_{\ge m}$ and $x_0$ with $\varphi^m(x_0)$ and assume that the orbit of $x_0$ is Zariski dense in $X$.  Now let $X_1,\ldots ,X_r$ denote the irreducible components of $X$.
Then there is some positive integer $M$ such that $\varphi^M(X_i)$ is Zariski dense in $X_i$ for $i=1,\ldots ,r$.  Then it suffices to prove that for $j\in \{0,\ldots ,M-1\}$ we have $\{n\colon f\circ \varphi^{Mn+j}(x_0)\in G\}$ is a finite union of arithmetic progressions along with a 
set of Banach density zero.  Since $\{\varphi^{Mn+j}(x_0)\colon n\ge 0\}$ is Zariski dense in some component $X_i$, we may replace $X$ by $X_i$, $x_0$ by $\varphi^j(x_0)$ and $\varphi$ with $\varphi^M$ and we may assume that $X$ is irreducible and that $\{\varphi^n(x_0)\colon  n\ge 0\}$ is Zariski dense in $X$.
Now let $N:=\{n \colon f(\varphi^n(x_0))\in G\}$.  If $N$ has Banach density zero, then there is nothing to prove.  Thus we may assume that $N$ has a positive Banach density, and hence it contains an infinite arithmetic progression, say $a\mathbb{N}_0+b$ with $a>0$.

We point out that the Zariski closure, $Y$, of the set $\{\varphi^{an+b}(x_0)\colon n\ge 0\}$ must be Zariski dense in $X$, since the union of the closures $Y_i$ of $\varphi^i(Y)$ for $i=0,1,\ldots ,a-1$ contains all but finitely many points in the orbit of $x_0$ and hence is dense in $X$.  Since $X$ is irreducible, we then see that $Y_i$ must be $X$ for some $i$, which then gives that $Y=X$.

Now for each $i\geq 0$, define a rational function $f_i:=f\circ\varphi^{ai}$, and set
\[ X_G := \{ x\in X : f_0(x),\ldots,f_{d}(x)\in G\}, \]
where $d$ is the dimension of $X$.
Then $X_G$ contains $\{\varphi^{an+b}(x_0)\colon n\ge 0\}$, which is Zariski dense in $X$ and so Lemma \ref{dichotomy} gives that the functions $f_0,\ldots,f_{d}$ satisfy some multiplicative dependence of the form
$$f_0^{i_0}\cdots f_d^{i_d}=c$$ with $c$ nonzero and $i_0,\ldots ,i_d$ integers with $\gcd(i_0,\ldots ,i_d)=1$.  It follows that if we set
$u_n = f(\varphi^n(x_0))$ then
$u_n^{i_0} u_{n+a}^{i_1}\cdots u_{n+ad}^{i_d}=c$ for every $n\ge 0$.  Moreover, by assumption the set of $n$ for which $u_n=0$ has Banach density zero and thus by Lemma \ref{lem:un}, the set
\[ \{n\in \N_0 : u_n\in G\} \]
is eventually periodic. This completes the proof.
\end{proof}
\smallskip
\section{Heights of points in orbits}\label{sec:Height}
Corollary \ref{orbit} gives an interesting ``gap'' about heights of points in the forward orbit of a self-map $\varphi$ for varieties and maps defined over $\Q$.

In order to state this gap result, we must first recall the definition of the Weil height here. Let $K$ be a number field and let $M_K$ be the set of places of $K$. For a place $v$, let $|\cdot|_v$ be the corresponding absolute value, normalized so that $|p|_v=p^{-1}$ when $v$ lies over the $p$-adic valuation on $\Q$. Let $K_v$ be the completion of $K$ at a place $v$ and let $n_v:=[K_v:\Q_v]$. Now define a function $H:\cl{\Q}\rightarrow [1,\infty)$ as follows: for $x \in \cl{\Q}\setminus \{0\}$, choose any number field $K$ containing $x$, and set
\[ H(x)^{[K:\Q]} := \prod_{v\in M_K} \max\{|x|_v^{n_v},1\}. \]
This is independent of choice of $K$ and defines a function $H:\cl{\Q}\rightarrow [1,\infty)$ called the \textit{absolute Weil height}. We let $h:\cl{\Q}\rightarrow [0,\infty)$ be its logarithm; i.e., $h(x):=\log H(x)$. For further background on height functions, we refer the reader to \cite[Chapter~2]{BG06} and \cite[Chapter~3]{GTM241}.
We have the following result.
\begin{thm}
Let $X$ be an irreducible quasiprojective variety with a dominant self-map $\varphi:X\to X$ and let $f:{X}\to \mathbb{P}^1$ be a rational map, all defined over $\bar{\Q}$.  Suppose that $x\in X$ has the following properties:
\begin{enumerate}
\item every point in the orbit of $x$ under $\varphi$ avoids the indeterminacy loci of $\varphi$ and $f$;
\item there is a finitely generated multiplicative subgroup $G$ of $\bar{\Q}^*$ such that $f\circ \varphi^n(x)\in G$ for every $n\in \mathbb{N}_0$.
\end{enumerate}
Then if $h(f\circ \varphi^n(x)) = {\rm o}(n^2)$ then the sequence
$(f\circ \varphi^n(x))_n$ satisfies a linear recurrence.  More precisely, there exists an integer $L\ge 1$ such that for each $j\in \{0,\ldots, L-1\}$ there are $\alpha_j,\beta_j\in G$ such that for all $n$ sufficiently large we have
$$f\circ \varphi^{Ln+j}(x) = \alpha_j \beta_j^n.$$
\label{orbit2}
\end{thm}
To prove this result, we first require an elementary estimate.
\begin{lemma}
\label{change of basis}
Let $K$ be a number field and let $G$ be a finitely generated free abelian subgroup of $ K^* \subset \bar{\Q}^*$ with multiplicative basis $g_1,\ldots,g_r$ for $G$.
Then there exists some positive constant $C$ such that
for each $a=g_1^{k_1}\cdots g_r^{k_r}\in G$ with $k_1,\ldots,k_r\in \Z$, we have the estimate
\[ h(a) \geq \max_{1\le i\le s}\, C|k_i|. \]
\end{lemma}
\begin{proof}
Since $G$ is a finitely generated subgroup of $K^*$ there exists a finite set $S$ of places of $K$ such that $|g|_v=1$ for every $g\in G$ whenever $v\not\in S$.
Let $v_1,\ldots ,v_s$ denote the elements of $S$.  Then for $i\in \{1,\ldots ,s\}$ we have a group homomorphism
$$\Psi_i : G\to \mathbb{R}$$ given by $g\mapsto \log\,|g|_{v_i}$.  Then there is a linear form $L_i(x_1,\ldots ,x_r)$ such that for $a=g_1^{k_1}\cdots g_r^{k_r}\in G$,
$$\Psi_i(a)=L_i(k_1,\ldots ,k_r).$$  We claim that $h(a)\ge |L_i(k_1,\ldots, k_s)|$ for $i=1,\ldots ,s$.  To see this, fix $i\in \{1,\ldots ,s\}$.
Since $\Psi(a^{-1})=-\Psi(a)$ and since $h(a)=h(a^{-1})$, we may assume without loss of generality that $|a|_{v_i}\ge 1$ and so $L_i(k_1,\ldots ,k_s)\ge 0$. Then
\begin{align*}
    h(a) &= \frac{1}{[K:\Q]}\sum_{v\in M_K} \log \max\{|a|_v^{n_v},1\} \\
         &\geq \frac{1}{[K:\Q]} \log\max\{|a|_{v_i}^{n_{v_i}},1\} \\
         &= \frac{n_{v_i}}{[K:\Q]}  L_i(k_1,\ldots ,k_s).
\end{align*}
Thus there is a positive constant $\kappa$ such that $h(a)\ge \kappa \cdot |L_i(k_1,\ldots ,k_s)|$ for $i=1,\ldots ,s$ and so \[ h(a) \geq \kappa\cdot \max_{1\le i\le s}(|L_i(k_1,\ldots, k_s)|). \]
By construction, the homomorphism $\Psi:G \to \mathbb{R}^s$ given by $g\mapsto (\Psi_i(g))_{1\le i\le s}$ is injective and so the image has rank $r$.  Thus after reindexing, we may assume that $L_1,\ldots ,L_r$ are linearly independent over $\mathbb{Q}$ and so there exist real constants $c_{i,j}$ for $1\le i,j\le r$ such that
$$\sum_{j=1}^r c_{i,j} L_j(x_1,\ldots ,x_s) = x_i$$ for $i=1,\ldots ,r$.  In particular, since for a given $i$ the $c_{i,j}$ cannot all be zero, there is some $C>0$ such that
$$0\neq \sum_{j=1}^r |c_{i,j}| \kappa^{-1} < C^{-1}$$ for $i=1,\ldots ,r$.
Then for $a=g_1^{k_1}\cdots g_r^{k_r}$, we have
\begin{align*}
   |k_i| &= \left| \sum_{j=1}^r c_{i,j} L_j(k_1,\ldots ,k_s)\right| \\
         &\leq \sum_{j=1}^r |c_{i,j}|\cdot |L_j(k_1,\ldots ,k_s)| \\
         &\leq \left(\sum_{j=1}^r |c_{i,j}|\kappa^{-1}\right) h(a)
\end{align*}
Thus $h(a)\ge C|k_i|$ for $i=1,\ldots ,r$, as required.
\end{proof}

\begin{proof}[Proof of Theorem \ref{orbit2}]
Let $g_1,\ldots ,g_d, g_{d+1},\ldots ,g_m$ be generators for $G$ so that $g_1,\ldots ,g_d$ generate a free abelian group and $g_{d+1},\ldots , g_{m}$ are roots of unity.  By Corollary \ref{orbit},
there are a positive integer $L$ and integer-valued sequences $b_{i,j}(n)$ for $j=0,\ldots ,L-1$ and $i=1,\ldots ,m$, each of which satisfies a linear recurrence, such that
$$f\circ \varphi^{Ln+j}(x) = \prod_i g_i^{b_{i,j}(n)}$$ for $n\ge p$.
Then multiplication by a root of unity does not affect the height of a number and so
$$h(f\circ \varphi^{Ln+j}(x))= h\left(\prod_{i=1}^d g_i^{b_{i,j}}\right).$$
Then if $h(f\circ \varphi^{Ln+j}(x))={\rm o}(n^2)$ then by Lemma \ref{change of basis} we must have
$b_{i,j}(n)={\rm o}(n^2)$ for $j=0,\ldots ,L-1$ and $i=1,\ldots ,d$.  Since it also is an integer-valued sequence satisfying a linear recurrence, we have that it is in fact ${\rm O}(n)$ and is ``piecewise linear'';
\textit{i.e.} it has the form $A+Bn$ on progressions of a fixed gap \cite[Proposition 3.6]{BNZ}.
Formally, this means that there exists a fixed $M\geq 1$ and integers $A_{i,j}, B_{i,j}$ for $j\in \{0,\ldots ,M-1\}$ and $i\in \{1,\ldots ,d\}$, and integer-valued sequences $c_{i,j}(n)$, which satisfy a linear recurrence for $i=d+1,\ldots ,m$ and $j=0,\ldots , M-1$, such that for $n$ sufficiently large we have
$$f\circ \varphi^{Mn+j}(x) = \prod_{i=1}^d g_i^{A_{i,j} + B_{i,j} n} \prod_{i=d+1}^m g_i^{c_{i,j}(n)}.$$
Since the $g_i$ are roots of unity for $i=d+1,\ldots ,m$ and since integer-valued sequences satifying a linear recurrence are eventually periodic modulo $N$ for every positive integer $N$, we see that for $n$ sufficiently large,
$f\circ \varphi^{Mn+j}(x)$ has the form $$\alpha_j \beta_j^n \omega^{t_j(n)},$$ where $\omega$ is a fixed $N$-th root of unity for some $N\ge 1$, $t_j(n)$ is eventually periodic, and $\alpha_j,\beta_j\in G$ and depend only on $j$ and not on $n$.  Since we only care about what holds for $n$ sufficiently large, it is no loss of generality to assume that each $t_j(n)$
is periodic and we let $p$ be a positive integer that is a common period for each of $t_0,\ldots ,t_{M-1}$. Then for $j\in \{0,\ldots ,M-1\}$ and $i\in \{0,\ldots ,p-1\}$ we have $$f\circ \varphi^{pM n+ Mi+j}(x)  = \left(\alpha_j \omega^{t_j(i)}\beta_j^i\right) ( \beta_j^p)^n.$$
The result now follows.
\end{proof}

\section{Applications to $D$-finite power series}\label{sec:Dfin}
In this section we apply our results to $D$-finite power series, showing how Theorem \ref{main theorem} generalizes a result of Methfessel \cite{Meth} and B\'ezivin \cite{Bez2}. We also look at classical results of P\'olya \cite{Pol} and B\'ezivin \cite{Bez} through a dynamical lens.

\begin{defn}
Let $K$ be a field.
A power series $F(x)\in K[[x]]\subseteq K((x))$ is \textit{$D$-finite} if the set of formal derivatives $\{F^{(i)}(x):i\geq 0\}$ is linearly dependent over
$K(x)\subseteq K((x))$; equivalently, $F(x)$ satisfies a linear differential equation of the form
\[ \sum_{i=0}^d p_i(x)F^{(i)}(x)=0 \]
where $p_0(x),\ldots,p_d(x)\in K[x]$ are polynomials, not all zero.

A sequence $(a_n)\in K^{\N_0}$ is \textit{holonomic} or \textit{P-recursive} over $K$ if it satisfies a recurrence relation
\[ a_{n+1} = \sum_{i=0}^d r_i(n)a_{n-i}, \]
for all $n\geq d$, where $r_0(t),\ldots,r_d(t)\in K(t)$ are rational functions. If each $r_i(t)$ is constant, then $(a_n)$ satisfies a \textit{$K$-linear recurrence}.
\end{defn}

For further background, we refer the reader to the works of Stanley \cite{Stanley, Stan}. It is well-known that if $F(x)=\sum_{n\geq 0}a_n x^n$ is a formal power series with coefficients in a field $K$ of characteristic zero, then $F(x)$ is $D$-finite (resp. rational) if and only if $(a_n)$ is $P$-recursive (resp. $K$-linearly recursive) \cite{Stanley}. The first application of our main result Theorem \ref{main theorem} in this setting is immediate, as follows.

Since $F(x)$ is $D$-finite, its coefficient sequence is $P$-recursive: there is a recurrence relation
\[ a_{n+1} = \sum_{i=0}^d r_i(n) a_{n-i}, \]
valid for all sufficiently large $n$, where the $r_i(x)\in K(x)$ are rational functions \cite{Stanley}. Thus we may define a rational map $\varphi:\mathbb{A}^{d+1}\dashrightarrow \mathbb{A}^{d+1}$ as follows:
\[ (t,t_1,\ldots,t_d) \mapsto \left(t+1, t_2,\ldots,t_d,\sum_{i=0}^d r_i(t)t_i \right).  \]
Here $(t,t_1,\ldots,t_d)$ are coordinates on $\mathbb{A}^{d+1}$. Now there is some $p>0$ such that none of the $r_i(x)$ have a pole at $x=n$ when $n\ge p$.  Now take the initial point to be $x_0:=(p,a_p,\ldots,a_{p+d-1})$ and the rational function $f(t,t_1,\ldots,t_d):=t_1$. Then the sequence $(a_n)_{n\geq 0}$ can be recovered as
\begin{equation} a_{n+p} = f(\varphi^n(x_0))\qquad ~{\rm for~}n\ge 0. \label{eq:Dfin}
\end{equation}
\begin{proof}[Proof of Theorem \ref{thm:Dfinite}]
By Equation (\ref{eq:Dfin}), after taking a suitable shift of the sequence $(a_n)_{n\geq 0}$, it can be recovered as
\[ a_n = f(\varphi^n(x_0)). \]
Thus the desired sets $N$ and $N_0$ are just
\[ N=\{n\in \N_0 : f(\varphi^n(x_0))\in G\}~~ {\rm and}~~ N_0=\{n\in \N_0 : f(\varphi^n(x_0))\in G\cup\{0\}\}. \]
Then we obtain the desired decomposition of $N$ from Theorem \ref{main theorem}; since
$N_0=N\cup Z$, where $Z=\{n\in \N_0 : f(\varphi^n(x_0))=0\}$, applying \cite[Theorem 1.4]{BGT} then gives that $Z$ is a finite union of arithmetic progressions along with a set of Banach density zero.  Then since both $N$ and $Z$ are expressible as a finite union of infinite arithmetic progressions along with 
a set of Banach density zero, so is their union.  The result follows.
\end{proof}
\subsection{Theorems of P\'olya and B\'ezivin.}
P\'olya \cite{Pol} showed that, given a fixed set of prime numbers $S$, if $F(x)=\sum a_n x^n\in \Z[[x]]$ is the power series of a rational function and the prime factors of $a_n$ lie inside of $S$ for every $n$, 
then there is some natural number $L$ such that for $n$ sufficiently large $$a_{Ln+j}=\frac{A_j}{B_j}\cdot \beta_j^n$$ where $A_j,B_j,$ and $\beta_j$ are integers whose prime factors lie inside of $S$ for $j=0,\ldots ,L-1$ and $B_j$
divides $A_j \beta_j^m$ for some positive integer $m$.  This result was later extended by B\'ezivin \cite{Bez}, who showed that if $K$ is a field of characteristic
zero and $G\le K^*$ is a finitely generated group then if $F(x)=\sum a_n x^n$ is a $D$-finite power series such that there is some fixed $m$ such that each $a_n$
is a sum of at most $m$ elements of $G$, then $F(x)$ is rational; moreover, he gave a precise form of these rational functions.
We show how to recover B\'ezivin's theorem in the case that $m=1$ from the dynamical results we obtained in the preceding sections. In particular, this recovers P\'olya's theorem.  We conclude by showing the relationship between these classical theorems and the dynamical results developed in the preceding sections.  More precisely, we give a dynamical proof of the following result.
\begin{thm} \label{thm:B} (B\'ezivin \cite{Bez}) Let $K$ be a field of characteristic
zero and let $F(x)=\sum a_n x^n\in K[[x]]$ be a $D$-finite power series such that $a_n\in G\cup \{0\}$ for every $n$,
where $G$ is a finitely generated subgroup of $ K^*$.  Then $F(x)$ is rational.
\end{thm}

To do this, we require a basic result on orders of zeros and poles of coefficients in a $D$-finite series. We recall that if $X$ is a smooth irreducible projective curve over an algebraically closed field $k$, and if $k(X)$ is the field of rational functions on $X$, then to each $p\in X$ we have a discrete non-archimedean valuation $\nu_p:k(X)^*\to \Z$ that gives the order of vanishing of a function at $p$ (when the function has a pole at $p$ then this valuation is negative).
Then for a function $f\in k(X)^*$ we have a divisor ${\rm div}(f) =\sum_{p\in X} \nu_p(f) [p]$, which is a formal $\Z$-linear combination of points of $X$.  The \emph{support} of ${\rm div}(f)$ is the (finite) set of points $p$ for which $\nu_p(f)\neq 0$; that is, it is the set of points where $f$ has a zero or a pole.  We make use of the fact $\sum_p \nu_p(f) = 0$ \cite[II, Corollary 6.10]{Hart}.

\begin{lemma} Let $E$ be an algebraically closed field of characteristic zero and let $K$ be the field of rational functions of
a smooth projective curve $C$ over $E$.  Suppose that $F(x)=\sum a_n x^n\in K[[x]]$ is $D$-finite, $a_n\neq 0$ for every $n$,
and that there is a finite subset $S$ of $C$ such that ${\rm div}(a_n)$ is supported on $S$ for every $n$.  Then for each $p\in S$, $\nu_p(a_n) =  {\rm O}(n)$.
\label{lem:pole}
\end{lemma}
\begin{proof} We have a polynomial recurrence $$R_M(n) a_{n+M} +\cdots + R_0(n) a_n =0$$ for $n$ sufficiently large.
Since each $R_i(n) = \sum_{j=0}^L r_{i,j} n^j$, we claim there is a fixed number $C_i$ such that $\nu_p(R_i(n))=C_i$ for sufficiently large $n$ for each nonzero polynomial $R_i$.
To see this, pick a uniformizing parameter $u$ for the local ring $\mathcal{O}_{X,p}$ and suppose that $Q(x)=q_0+\cdots + q_L x^L$ is a nonzero polynomial in $K[x]$.  Then we can rewrite it as
$\sum_{i=0}^N u^{m_i} q_i' x^L$ where $m_i=\nu_p(q_i)$ and $q_i'\in \mathcal{O}_{X,p}^*$.  Let $s$ denote the minimum of $m_0,\ldots, m_N$.  Then $Q(x)/u^s = \sum_{i=0}^N u^{m_i-s} q_i' x^L$ and by construction
$$\sum_{i=0}^N u^{m_i-s}(p) q_i'(p) x^L$$ is a nonzero polynomial in $E[x]$ and hence it is nonzero for $n$ sufficiently
large, which shows that $\nu_p(Q(n))=s$ for all $n$ sufficiently large.  Thus
in particular if $C$ is the maximum of the $C_i$ as $i$ ranges over the indices for which $R_i(x)$ is nonzero, then for $n$ sufficiently large
\begin{align*}
\nu_p(a_{n+M}) & =
\nu_p(R_M(n) a_{n+M}) - C\\
 & = \nu_p(\sum_{i=0}^{M-1} R_i(n) a_{n+i}) - C\\
&\ge  -2C + \min(\nu_p(a_{n+i}\colon i=0,\ldots ,M-1).
\end{align*}
It follows that $\nu_p(a_n) \ge -2C n + B$ for some constant $B$ for all sufficiently large $n$.  It follows that there is a fixed constant $C_0$ such that
$\nu_p(a_n) \ge -C_0 n$ for every $p\in S$, for all $n$ sufficiently large.  To get an upper bound, observe that
$\sum_{p\in S} \nu_p(a_n) = 0$ \cite[II, Corollary 6.10]{Hart} and so for $n$ sufficiently large we have
$$\nu_p(a_n) = \sum_{q\in S\setminus \{p\}} -\nu_q(a_n) \le (|S|-1)C_0n,$$ which now gives $\nu_p(a_n)={\rm O}(n)$.

\end{proof}
We now give a quick overview of how one can recover Theorem \ref{thm:B} from the above dynamical framework.
\begin{proof}[Proof of Theorem \ref{thm:B}] By Theorem \ref{thm:Dfinite}, $\{n\colon a_n\in G\}$ is a finite union of arithmetic
progressions along with a set of Banach density zero.  Hence there exists some $L$ such that for each $j\in \{0,\ldots ,L-1\}$, the set $\{n\colon a_{Ln+j}=0\}$ either contains all sufficiently large $n$ or it is a set of Banach density zero.  Since $F(x)=\sum a_n x^n$ is $D$-finite if and only if for each $L\ge 1$ and each $j\in \{0,\ldots ,L-1\}$,
the series $\sum a_{Ln+j} x^n$ is $D$-finite, by Lemma \ref{lem:un} it suffices to consider the case when $a_n\in G$ for every $n$.
The fact that the coefficients are $P$-recursive gives that there is a finitely generated field extension $K_0$ of $\mathbb{Q}$ such that $F(x)\in K_0[[x]]$.  We prove the result by induction on ${\rm trdeg}_{\Q}(K_0)$.  If $[K_0:\Q]<\infty$ then $K_0$ is a number field.  Then by \cite[Theorem 1.6]{BNZ}, $h(a_n)={\rm O}(n\log n)$ and by Equation (\ref{eq:Dfin}) and Theorem \ref{orbit2}, we then get $a_n$ satisfies a linear recurrence, giving the result when $K_0$ has transcendence degree zero over $\Q$.

We now suppose the the result holds whenever $K_0$ has transcendence degree less than $m$, with $m\ge 1$,
and consider the case when ${\rm trdeg}_{\Q}(K_0)=m$.  Then there is subfield $E$ of $K_0$ such that $K_0$
has transcendence degree $1$ over $E$ and such that $E$ is algebraically closed in $K_0$.  Since $K_0$ has
characteristic zero and $E$ is algebraically closed in $K_0$, $K_0$ is a regular extension of $E$, and so $R:=K_0\otimes_E \bar{E}$ is an integral domain.
Then the field of fractions of $R$ is the field of regular functions of a smooth projective curve $X$ over $\bar{E}$.
Now let $g_1,\ldots ,g_d, g_{d+1},\ldots ,g_m$ be generators for $G$ so that $g_1,\ldots ,g_d$ generate a free abelian group and $g_{d+1},\ldots , g_{m}$ are roots of unity
and let
$\{p_1,\ldots ,p_{\ell}\}\in X$ denote the collection of points at which some element from $g_1,\ldots ,g_d$ has a zero or a pole.
Then there are integers $b_{i,j}$ such that
$${\rm div}(g_i) = \sum b_{i,j} [p_j]$$ for $i=1,\ldots ,d$.
Now we have $ a_n = g_1^{e_1(n)}\cdots g_m^{e_m(n)}$ and so
$${\rm div}(a_n) = \sum_{j=1}^{\ell} \left( \sum_{i=1}^d b_{i,j} e_i(n)\right) [p_j].$$
In particular, by Lemma \ref{lem:pole}, $$\sum_{i=1}^d b_{i,j}  e_i(n) = {\rm O}(n)$$ and since the left-hand side satisfies a linear recurrence, we have that it is piecewise linear in the sense of having the form $A+Bn$ on progressions of a fixed gap \cite[Proposition 3.6]{BNZ}.  
It then follows that there exist some $r\ge 1$ and some fixed $h_0,\ldots ,h_{r-1}\in K_0^*$ such that for $j\in \{0,\ldots ,r-1\}$,
$ a_{r(n+1)+j}/a_{rn+j} = C_{j,n} h_j$, where $C_{j,n}\in E^*$.
It follows that for $ a_{rn+j} = C_{j,n} h_j^n P_j$, where $P_j\in K_0$ is constant.  Then since the series $\sum P_j^{-1} h_n^{-n} x^n$ is $D$-finite and since $D$-finite series are closed under Hadamard product,
$$G_j(x):=\sum C_{j,n} x^n \in \bar{E}[[x]]$$ is $D$-finite and takes values in a finitely generated multiplicative group.  Thus $G_j(x)$ is rational and then it is straightforward to show that
$$F_j(x):=\sum  a_{rn+j} x^n=\sum C_{j,n} P_j h_j^n x^n$$ must also be rational and thus $F(x)=\sum_{j=0}^{r-1} x^jF_j(x^r) $ is also rational, as required.
\end{proof}


\end{document}